\documentclass[twoside,reqno,12pt]{amsart}
\usepackage[left=1in,right=1in,top=1in,bottom=1in]{geometry}

\usepackage[colorlinks=true]{hyperref}

\usepackage{tikz}
\usepackage{amsmath}
\usetikzlibrary{decorations.pathmorphing,calc}

\usepackage{amsfonts}
\usepackage{amssymb}
\usepackage{color}
\usepackage{xcolor}
\usepackage{comment}
\usepackage{enumerate}
\newcommand{\HOX}[1]{\marginpar{\footnotesize sharp1}}

\newtheorem{thm}{Theorem}[section]

\newtheorem{lem}[thm]{Lemma}
\newtheorem{prop}[thm]{Proposition}
\newtheorem{exmp}[thm]{Example}
\newtheorem{defn}[thm]{Definition}
\newtheorem{rem}[thm]{Remark}

\theoremstyle{definition}

\numberwithin{equation}{section}

\newcommand{\thmref}[1]{Theorem~\ref{sharp1}}
\newcommand{\secref}[1]{\S\ref{sharp1}}
\newcommand{\lemref}[1]{Lemma~\ref{sharp1}}

\renewcommand{\t}[1]{\tilde{sharp1}}

\newcommand{\pM}{{\partial M}}

\newcommand{\dbtilde}[1]{\tilde{\tilde{sharp1}}}
\def\hat{\widehat}
\def\tilde{\widetilde}
\def \bfo {\begin {eqnarray*} }
\def \efo {\end {eqnarray*} }
\def \ba {\begin {eqnarray*} }
\def \ea {\end {eqnarray*} }
\def \beq {\begin {eqnarray}}
\def \eeq {\end {eqnarray}}
\def \supp {\hbox{supp }}

\def \dist {\hbox{dist}}

\def \det {\hbox{det}}

\def \e {\varepsilon}

\def\hat{\widehat}
\def\tilde{\widetilde}
\def \bfo {\begin {eqnarray*} }
\def \efo {\end {eqnarray*} }
\def \ba {\begin {eqnarray*} }
\def \ea {\end {eqnarray*} }
\def \beq {\begin {eqnarray}}
\def \eeq {\end {eqnarray}}
\def \supp {\hbox{supp }}

\def \dist {\hbox{dist}}

\def \det {\hbox{det}}

\def \e {\varepsilon}

\newcommand{\intM}{M^{{o}}}
\newcommand{\Char}{\mathrm{Char}}

\newcommand{\uinc}{u^\mathrm{inc}}
\newcommand{\uref}{u^\mathrm{ref}}

\usepackage{mathtools}

\newcommand{\tM}{M_e}

\newcommand{\xzp}{x_0'}
\newcommand{\xkp}{x_k'}
\newcommand{\xikp}{{\xi^k}'}
\newcommand{\xik}{{\xi^k}}

\newcommand{\mR}{\mathbb{R}}

\newcommand{\mN}{\mathcal{N}}

\usepackage{color,comment}

\newcommand{\xizp}{{\xi^0}'}
\newcommand{\xiz}{{\xi^0}}
\newcommand{\mH}{\mathcal{H}}

\newcommand{\blue}{\color{blue}}

\begin{document}
\title[Partial Data Dirichlet-to-Neumann Map on Lorentzian Manifold]{The Dirichlet-to-Neumann map for Lorentzian Calder\'on problems with data on
disjoint sets}

\author[]{Yuchao Yi, Yang Zhang}
\address{}
\email{}


\begin{abstract}
    We consider the restricted Dirichlet-to-Neumann map \( \Lambda^{U,V}_{g,A,q} \) for the wave equation with magnetic potential \( A \) and scalar potential \( q \), on an admissible Lorentzian manifold \( (M, g) \) of dimension \( n \geq 3 \) with boundary.
    Here \( U \) and \( V \) are disjoint open subsets of \( \partial M \), where we impose the Dirichlet data on \( U \) and measure the Neumann-type data on \( V \).
    We use the gliding rays and microlocal analysis to show that, without any a priori information, one can reconstruct the conformal class of the boundary metric \( g|_{T\partial M \times T\partial M} \) and the magnetic potential \( A|_{T\partial M} \) at recoverable boundary points from \( \Lambda^{U,V}_{g,A,q} \). In particular, the conformal factor and the jet of the metric at those points are determined up to gauge transformations. Moreover, if the metric and the time orientation are known on \( U \) (or \( V \)), then the metric on a larger portion of \( V \) (or \( U \)) can be reconstructed, up to gauge.
\end{abstract}

\maketitle

\tableofcontents

\section{Introduction}
We consider a time-oriented Lorentzian manifold $(M,g)$ of dimension $n$ for $n\ge3$, with a {timelike, strictly null-convex} boundary $\partial M$, see Definition \ref{def: admissible manifold}.
Suppose $g$ has the signature $(-1, 1, \ldots, 1)$.
We consider the wave operator with a first-order perturbation
\[
\Box_{g,A,q} \;=\;
\frac{1}{\sqrt{|\det g|}}
\bigl(\partial_j - i A_j \bigr)\Bigl(\sqrt{|\det g|}\,g^{jk}\bigl(\partial_k - i A_k\bigr)\Bigr)
\;+\; q,
\]
where $A$ is a smooth 1‐form and $q$ is a smooth scalar potential on $M$.
Let $u$ be the outgoing solution of
\begin{equation}\label{eq: wave equation}
        \begin{aligned}
            \Box_{g,A,q} u &= 0 \quad \text{in } M^{\circ}, \\
            u &= f \quad \text{on } \pM,\\
            \text{supp}(u) &\subseteq J^+(\text{supp}(f)),
        \end{aligned}
\end{equation}
where $J^+(\text{supp}(f))$ denotes the causal future of support of $f$, for more details see Section \ref{sec: preliminaries}.
Here by outgoing, we mean the microlocal solution that we choose has singularities propagating along forward bicharacteristics, for more details see Section \ref{subsec: gliding}.
The associated outgoing \emph{Dirichlet-to-Neumann map} (DN map) is defined as
\[
\Lambda_{g,A,q}: f \mapsto (\partial_\nu u-i\langle A, \nu\rangle u)|_{\partial M},
\]
where $\nu$ denotes the unit outward normal vector field to $\pM$.
%
Let $U$ and $V$ be open subsets of $\partial M$.
The goal of this work is to study the inverse problem of recovering $g$ and $A$ up to gauge transformations, from partial boundary data.
More precisely, we impose the Dirichlet boundary conditions $u= f$ on $U$ and measure the corresponding {Neumann‐type} derivative on $V$.
For this purpose, we define the \emph{restricted DN map} as
\[
\Lambda^{U,V}_{g,A,q}: f \mapsto (\partial_\nu u-i\langle A, \nu\rangle u)|_{V} = \partial_{\nu}u|_V - i(\langle A, \nu \rangle f)|_{U \cap V}.
\]
Throughout the paper, we shall mainly focus on the case where $U \cap V = \emptyset$, unless stated otherwise. In this case the DN map is simply given by
\[
\Lambda^{U,V}_{g,A,q}(f) = \partial_{\nu}u|_V.
\]

Boundary determination problems have been studied in Riemannian settings, with well-established results on uniqueness, stability, and constructive reconstruction methods, see \cite{lassas2003semiglobal, stefanov2005boundary, uhlmann2003boundary} for cases with strictly convex boundaries, and \cite{vargo2009lens, herreros2011scattering} for cases with concave points and analytic structures.
In particular, for smooth metrics in dimensions $n\geq 3$, not only the full jet of the metric at the boundary but also the metric in the interior near a strictly convex point can be recovered, which leads to global rigidity results under foliation conditions, see \cite{stefanov2016boundary,uhlmann2016inverse, stefanov2021local}.
Besides, there are also many other works addressing lens and scattering rigidity, including \cite{burago2010boundary, croke2014scattering,croke2004rigidity,croke2000local,croke2011lens,guillarmou2017lens,michel1981rigidite,muhometov1981problem,muhometov1978problem,pestov2005two,stefanov1998rigidity,stefanov2009local}.


In contrast, the case for more general Lorentzian metrics remains less understood, with fewer results on boundary determination with partial data.
The recovery of stationary metrics from the time separation function is studied in
\cite{andersson1996boundary} for two-dimensional product Lorentzian manifolds, and in \cite{lassas2016determination} for universal covering spaces of real-analytic Lorentzian manifolds, in \cite{uhlmann2021travel} for manifolds with dimensions three and higher.
The recovery of stationary metrics from scattering relation (also referred to as lens relation) is studied in \cite{stefanov2024lorentzian}, where the problem is reduced to boundary or lens rigidity of magnetic systems on the base via null-geodesics. Then it is considered in \cite{munoz2024} via timelike geodesics and MP-systems.
Most recently, it is proved in \cite{stefanov2024boundary} that the jets of smooth Lorentzian metrics on the boundary near a lightlike strictly convex point can be recovered from the scattering relation, up to a gauge transformation.
The scattering relation of null-geodesics between two Cauchy surfaces for metrics with small perturbation is considered in \cite{wang2023rigidity}.

These Lorentzian rigidity problems are closely related to the so-called Lorentzian Calder\'{o}n  problems, which are analogues for wave equations of the classical Calder\'{o}n  problems for elliptic equations.
In a classical Calder\'{o}n  problem, we consider the recovery of a conductivity or the Riemannian manifold $(M,g)$ from the DN map associated with a Laplace-Beltrami equation. 
Results for real-analytic Riemannian metrics can be found in \cite{lassas2001determining, lassas2003dirchlet,lee1989determining} and for metrics with fixed conformal structure can be found in \cite{dos2009limiting, ferreira2016calderon,carstea2023remarks}. 
For more results, see, for example, surveys \cite{GUsurvey, MR3098658, MR3221605}. 
In a Lorentzian Calder\'{o}n  problem, we consider the recovery of the Lorentzian manifold $(M,g)$ from the DN map associated with a wave equation.
For ultra-static manifold with stationary metrics or metrics real-analytic in $t$,
this problem has been extensively studied in the literature, including
\cite{gel1954some,alinhac1983non,belishev1987approach,eskin2007inverse,kachalov2001inverse,kurylev2018inverse,lassas2018inverse,lassas2014inverse,lassas2010inverse},
based on Gelfand problems, inverse spectral methods, or the Boundary Control method.
In \cite{alexakis2022} and \cite{alexakis2024lorentzian}, the assumption of real-analyticity in $t$ is replaced by bounds on the Lorentzian curvature and the Lorentzian Calder\'{o}n problem was solved in a fixed conformal class.
In \cite{oksanen2024rigidity}, the global determination using smaller amount of measurements from the DN map is proved, when the metric is globally hyperbolic in $\mR^{1+n}$ and agrees with the Minkowski one outside a compact set.

While interior determination in a general Lorentzian manifold remains open, boundary determination has been studied in \cite{Stefanov2018}.
Such boundary determination is typically the first step towards interior recovery, as it provides essential geometric information on the boundary.
It is shown in \cite{Stefanov2018} that for an admissible Lorentzian manifold, the DN map determines the jets of $g, A, q$ on the boundary up to a gauge transformation in a stable way.
Here the reconstruction is local or semilocal, but with sources and measurements on the same sets.
This allows one to use the pseudodifferential part of the DN map for the reconstruction.
In this work, we consider boundary determination from the DN map, where the sources and measurements are supported in disjoint subsets of the boundary.
This setting models scenarios in mathematical physics and applied fields, where full boundary access is not available.
For example, in geophysical imaging, sensors may need to be placed far from the sources due to physical constraints, such as in oil exploration as mentioned in \cite{lassas2014inverse}.
These partial data problems pose additional challenges.
Our goal is to investigate what geometric information, particularly jets of the metric and the magnetic potential on the boundary, can still be recovered from such incomplete measurements.



\subsection*{Main results}
In this paper, we consider admissible Lorentzian manifolds defined as follows.
\begin{defn}
    Let $n \geq 3$ and $(M,g)$ be a smooth connected $n$-dimensional Lorentzian manifold with non-empty boundary $\pM$, where $g$ has the signature $(-,+,\dots,+)$. We say $(M,g)$ is {admissible} if
    \begin{enumerate}
        \item there exists a smooth, proper, surjective function $\tau \colon M \to \mathbb{R}$, such that $d\tau$ is everywhere timelike;
        \item $\partial M$ is timelike, i.e., the induced metric $\bar{g} := g|_{T\partial M \times T\partial M}$ is Lorentzian;
        \item $\partial M$ is \emph{strictly null-convex}, i.e., if $\nu$ denotes the outward pointing unit normal vector field on $\partial M$, then
        \[
        II(V,V) = g(\nabla_V \nu, V) > 0 \tag{1.1},
        \]
        for all null vectors $V \in T_p \partial M$, where $II$ is the second fundamental form.
    \end{enumerate}
\end{defn}
We call such $\tau$ a temporal function.
The definition above is adapted from \cite[Definition 1.1]{Hintz2017}, with some minor differences: here we assume the temporal function is smooth and the boundary is strictly null-convex.
With the existence of the temporal function, the well-posedness of \eqref{eq: wave equation} follows from \cite[Theorem 24.1.1]{MR2304165}, see also \cite{alexakis2022}.

Recall $T^*\partial M \setminus 0$ is naturally decomposed into the elliptic, glancing, and hyperbolic sets.
More explicitly, the glancing set contains lightlike covectors in $T^*\pM$ and the hyperbolic set contains timelike covectors in $T^*\pM$.
The glancing set plays an important role in our analysis,
which can be further decomposed into the diffractive part, the gliding part, and higher-order glancing sets, according to the behavior of the null-bicharacteristics near the boundary.
Gliding rays are corresponding to the gliding part.
Roughly speaking, they are generalized bicharacteristics that hit the boundary tangentially, with second order contact, lying locally in $\pM$.
For more details see Section \ref{subsec: gliding}.
Now let $S^*\partial M \coloneqq (T^*\partial M \backslash 0) / \mathbb{R}^+$.
As we can only hope to determine the geometry at boundary points that are visible to measurements,
we introduce the following definition for covectors in the glancing set. For definition of forward/backward gliding rays, see Section \ref{subsec: gliding}.

\begin{defn}\label{def: admissible manifold}
    Let $(M, g)$ be an admissible Lorentzian manifold with dimension $n \geq 3$.
     Let $\mathcal{G}$ be the glancing set defined as in (\ref{def: glancingset}).
     A glancing covector $(x', \xi') \in \mathcal{G}$, viewed as an element of $S^*\partial M$,  is called a {recoverable direction} if
    \begin{itemize}
        \item $x' \in U$ and the corresponding forward gliding ray passes through $V$; or
        \item $x' \in V$  and  the corresponding backward gliding ray passes through $V$.
    \end{itemize}
    The point $x'$ is called a recoverable point if
    \begin{itemize}
        \item for $n = 3$,  both directions in glancing set are recoverable;
        \item for $n > 3$, at least one direction in glancing set is recoverable.
    \end{itemize}
    We denote by $\tilde{U}$ the set of recoverable points in $U$ and by $\tilde{V}$ the set of recoverable points in $V$.
    We refer to them as \emph{recoverable sets}.
\end{defn}
We remark that in our setting, the projection of a gliding ray onto the base manifold is in fact a null-geodesic on $\pM$ with respect to the induced metric $\bar{g}$, see Lemma \ref{lemma: gliding rays are tame}.
Let $\mathcal{L}^\pm_{\pM}(p)$ denote the future (or past) light cone from $p \in \pM$, defined as the union of all forward (or backward) null-geodesics in $(\pM, \bar{g})$ emanating from $p$.
For definition of forward/backward null-geodesics, see Section \ref{subsec: geometry}.
When $n >3$, the set $\tilde{U}$ consists of the points in $U$ whose the future light cone in $(\pM, \bar{g})$ intersects $V$.
Equivalently, these are points in $U$ that lie on the past light cone in $(\pM, \bar{g})$ of some point in $V$.
A similar argument applies to $\tilde{V}$.
Consequently, the recoverable sets are explicitly given by
\begin{align*}
 \tilde{U} = (\textstyle\bigcup_{q \in V} \mathcal{L}^-_{\pM} (q)) \cap U \quad \text{ and } \quad  \tilde{V} = (\bigcup_{p \in U} \mathcal{L}^+_{\pM} (p)) \cap V.
\end{align*}
In addition, we consider the future pointing lens relation $L$, also referred to as the scattering relation in the literature.
We describe it below, see also
Definition \ref{def: lens} and Remark \ref{rem: lens relation} about future pointing lens relation.
Let $(x',\xi') \in T^*\pM$ be timelike and future pointing.
Let $(x', \xi)\in T^*M$ be the unique inward-pointing lightlike covector with orthogonal projection $(x', \xi')$.
There exists a unique forward null-bicharacteristic $\Gamma$ starting from $(x', \xi)$.
If $\Gamma$ intersects the boundary again at some $(y', \eta) \in L^*M$,
then under our assumptions the intersection must be transversal.
Let $(y', \eta')\in T^*M$ be the orthogonal projection of $(y', \eta)$.
We define $L(x', \xi') = (y', \eta')$.
Using that $L$ is an even map, one can define $L$ for past pointing timelike covector, for more details see \cite{Stefanov2018}.



Recall  \(\bar{g}\)  denotes the restriction of \(g\) to \(T\partial M \times T\partial M\), i.e., \(\bar{g} = g|_{T\partial M \times T\partial M}\).
We assume no a priori information is known on the open disjoint boundary sets $U$ or $V$.
The following theorem summarizes what can be reconstructed from the restricted Dirichlet–to–Neumann map
$\Lambda^{U,V}_{g,A,q}$.
\begin{thm}\label{theorem: reconstruction}
    Let $(M, g)$ be an admissible Lorentzian manifold of dimension $n \geq 3$ and $U, V$ be disjoint open subsets of $\pM$.
    Then from the restricted DN map $\Lambda^{U,V}_{g,A,q}$, we can reconstruct
    \begin{itemize}
        \item[(1)] the recoverable sets $\tilde{U}$ and $\tilde{V}$;
        \item[(2)] the conformal class of $\bar{g}$ on $\tilde{U}$ and $\tilde{V}$;
        \item[(3)] for each recoverable direction $(x', \xi')$ on $U$ or $V$, the lens relation $L: \mN \rightarrow \mN$ up to time orientation, where $\mN$ is a small conic timelike neighborhood of $(x', \xi')$;
        \item[(4)] the magnetic potential $A|_{T\pM}$ on $\tilde{U}$ and $\tilde{V}$.
    \end{itemize}
\end{thm}
    We note that there is a natural obstruction to reconstructing $g, A, q$ from the restricted DN map $\Lambda^{U,V}_{g,A,q}$.
    Specifically, there exist gauge transformations that keep $\Lambda^{U,V}_{g,A,q}$ invariant, see Lemma \ref{lem: gauge2}.
    One can only expect to reconstruct these quantities up to such a gauge transformation.
    Without prior knowledge of $U$ or $V$,
    Theorem \ref{theorem: reconstruction} shows we can reconstruct the induced $\bar{g}$ on the boundary for points in the recoverable sets, up to conformal factors.
    Moreover, we can reconstruct the $1$-form $A$ on the boundary for these points, up to gauge transformations.
    Indeed, we may pick local coordinates $(x', x^n)$ near such a boundary point and suppose it is the semi-geodesic normal coordinates for some metric, see Lemma \ref{lem: semi-geodesic}.
    Then by Lemma \ref{lem: semi-geodesic}, there exists a boundary-preserving diffeomorphism $\Psi$ such that $\Psi^*g$ has $(x', x^n)$ as the semi-geodesic normal coordinates as well.
    Then by \cite[Lemma 2.5]{Stefanov2018}, in these coordinates we can modify the 1-form $\Psi^*A$ using a gauge transformation such that $(\Psi^*A)_n(x',0) = 0$.
    Note that $\Psi^*A|_{T\pM} = A|_{T\pM}$ as $\Psi$ fixes the boundary.
    This implies we reconstruct $\Psi^*A$ on the boundary and therefore $A$ on the boundary up to gauge transformations.

    One may extend the definition of lens relation by using broken null-bicharacteristics reflecting finitely many times at
    $\partial M$; see Definition \ref{def: brokengeo} and \ref{def: brokenbichar}.
    Indeed, let $(x', \xi') \in T^*\pM$ be defined as above.
    For consistency we write $(x', \xi')$ as $(\xzp, \xizp)$.
    We consider the unique broken null-bicharacteristic starting from $(\xzp,\xiz)$, reflecting successively at $\xkp \in \pM$ with $(\xkp, \xik)\in L^*M$, for $k = 1, \ldots, K$.
    Let $(\xkp,\xikp) \in T^*\pM$ be the orthogonal projection of $(\xkp, \xik)$.
    Then we define the $k$-th lens relation $L^k$ as $ L^k(\xzp, \xizp) = (\xkp, \xikp)$.
    By Proposition \ref{prop: principal symbol formula}, when $\xzp \in U$ and $\xkp \in V$, such $L^k$ can be reconstructed from the restricted DN map, although the exact value of $k$ may not be recovered), see Section \ref{sec: recovery of weak lens relation}.
\vspace{0.5em}

Based on this constructive result, we prove the following determination result.
\begin{thm}\label{theorem: determination}
    Let $(M, g_1)$ and $(M, g_2)$ be admissible Lorentzian manifolds  of dimension $n \geq 3$.
    Let $U, V$ be disjoint open subsets of $\pM$.
    Suppose
    \[
    \Lambda^{U,V}_{g_1,A_1,q_1} f = \Lambda^{U,V}_{g_2,A_2,q_2} f
    \]
    for any distribution $f$ with compact support in $U$.
    Then 
    $\bar{g}_1 = e^{\varphi_0} \bar{g}_2$ in the same recoverable sets $\tilde{U}$ and $\tilde{V}$,
    where $\varphi_0 \in C^{\infty}(\tilde{U} \cup \tilde{V})$ is locally constant.
    In particular,
    if $W_1 \subseteq \tilde{U}$ and $W_2 \subseteq \tilde{V}$ are connected components joined by some forward gliding ray or broken bicharacteristic, then there exists a constant $c$ such that
    \[
    \varphi_0|_{W_1} \equiv \frac{c}{n-2}\quad \text{ and }\quad \varphi_0|_{W_2} \equiv \frac{c}{n}.
    \]
    {Furthermore, there exists $\varphi \in C^\infty(M)$ with $\varphi|_{\pM} = \varphi_0$ and a local diffeomorphism $\psi$ near $\pM$ that fixes it pointwise, such that the jets of $g_1$ and $e^\varphi \psi^*g_2$ coincide in $\tilde{U}$ and $\tilde{V}$.
    }
\end{thm}
Compared to Theorem \ref{theorem: reconstruction}, this theorem is a uniqueness result that further describes the conformal factor between $\bar{g}_1$ and $\bar{g}_2$ when the corresponding restricted DN maps coincide.
Moreover, it proves that the jets of the metric on the boundary are determined up to gauge transformations at points in the recoverable sets.

Next, assume the induced metric and time orientation are known on $U$ (respectively on $V$).
Then the structure on a subset of $V$ (respectively of $U$) can be reconstructed.
To state the result, we introduce the following definition.
\begin{defn}
    We say $x'$ is $V$-recoverable, if there exists a forward broken null-geodesic from $x'$ to some $y' \in V$.
    Similarly, $y'$ is called $U$-recoverable, if there exists a backward broken null-geodesic from $y'$ to some $x' \in U$.
    We denote by $\tilde{U}_0$ the set of $V$-recoverable points in $U$ and by $\tilde{V}_0$ the set of $U$-recoverable points in $V$.
\end{defn}
More explicitly, let $\mathcal{L}^{\pm, b}(p)$ denote the union of all forward (or backward) broken null-geodesics starting from $p$ in $(M,g)$.
Then $\tilde{U}_0$ and $\tilde{V}_0$ can be written as
\[
\tilde{U}_0 = (\textstyle\bigcup_{q \in V} \mathcal{L}^{-, b} (q)) \cap U \quad \text{ and } \quad  \tilde{V}_0 = (\bigcup_{p \in U} \mathcal{L}^{+,b} (p)) \cap V.
\]
We emphasize that both $\tilde{U}_0$ and $\tilde{V}_0$ are open subsets, since under our assumptions broken null-geodesics are stable under small perturbations.
In particular, one has $\tilde{U} \subseteq \tilde{U}_0$ and $\tilde{V} \subseteq \tilde{V}_0$ by Lemma \ref{lemma: broken bichar converge to gliding on boundary}.
In this setting, we state the following theorem.

\begin{thm}\label{theorem: reconstruction with a priori info}
    Let $(M,g)$ be an admissible Lorentzian manifold of dimension $n \geq 3$ and $U, V$ be open disjoint subsets of $\pM$.
    Given the induced metric $\bar{g}$ and the time orientation on $\tilde{U}_0$,
    from the restricted DN map $\Lambda^{U,V}_{g,A,q}$, one can reconstruct $\bar{g}$ and the time orientation on $\tilde{V}_0$.
    The same result holds with $\tilde{U}_0$ and $\tilde{V}_0$ interchanged.
\end{thm}
    This theorem implies that with additional information on $U$,  we can reconstruct $g$ on the boundary at points in the $U$-recoverable sets of $V$,  up to gauge transformations.
    Indeed, with the induced metric $\bar{g}$ reconstructed in the $U$-recoverable sets, we may choose local coordinates $(x', x^n)$ near a boundary point.
    As before, by Lemma \ref{lem: semi-geodesic}, there exists a boundary-preserving diffeomorphism $\Psi$ such that $\Psi^*g$ has $(x', x^n)$ as the semi-geodesic normal coordinates.
    Since $\Psi^*g|_{T\pM \times T\pM} = \bar{g}$,
    we actually reconstruct $\Psi^*g$ on the boundary in  these coordinates and thus $g$ on the boundary up to gauge transformations.

\subsection{Idea of the proof}
The full data DN map $\Lambda_{g,A,q}$ on an admissible Lorentzian manifold is known to be associated with two canonical relations: the diagonal and the graph of the lens relation.
The first corresponds to its pseudodifferential operator part, and the second to its Fourier Integral Operator part.
In \cite{Stefanov2018}, the jet bundles of $g$, $A$ and $q$ on the boundary are stably recovered {up to gauge transformations} using the full symbol of the pseudodifferential part of $\Lambda_{g,A,q}$.
However, when the source region $U$ and the observation region $V$ are disjoint, we no longer have access to the pseudodifferential part of the DN map. Instead, we fully rely on the Fourier Integral Operator part of the DN map to recover information. We perform symbolic construction of solutions to \eqref{eq: wave equation} with conormal distribution $f$ to compute the principal symbol of $\Lambda^{U,V}_{g,A,q}$ on the Fourier Integral Operator part. We show that the absolute value of the principal symbol connects the determinant of the metric on $U$ and $V$ via broken bicharacteristics, while the argument of the principal symbol reveals information about the light ray transform of $A$ along broken null-geodesics.

In order to make use of the information provided by the principal symbol of the DN map, we utilize gliding rays, which are essentially null-bicharacteristics for the metric restricted to the boundary, see {Lemma \ref{lemma: gliding rays are tame}.}
Gliding rays are important in our proof for two main reasons:
\begin{enumerate}
    \item the union of gliding rays acts as light observation sets, which provides information about the structure of the light cone on the boundary;
    \item the gliding rays are entirely in $T^*\partial M$, in contrast to the interior, where we have less control on the boundary.
\end{enumerate}
The use of gliding rays also allows us to have little restriction on $U$ and $V$, for example,  they may arbitrarily small and away from each other.
The amount of information that can be recovered then only depends on the existence of gliding rays connecting them.

Therefore, we first use the propagation of singularities to identify gliding rays on the boundary.
From these gliding rays, we reconstruct the light cone and hence the conformal class of the metric at points in the recoverable sets on the boundary.
We prove that around the glancing set, we can also construct the lens relation up to time orientation.
This allows us to use the main result in \cite{stefanov2024boundary} to determine the jet bundle of the metric up to gauge transformations.

We then use the principal symbol of $\Lambda^{U,V}_{g,A,q}$ to determine the conformal factor and to construct the 1-form $A$ restricted to boundary tangential directions.
\begin{enumerate}
    \item The magnitude of the principal symbol gives information of the determinant of the metric. We use a geometric argument to show that the determinant relation provided by the magnitude of the principal symbol is enough to guarantee that the conformal factor is locally constant. This is a deterministic result.
    \item The argument of the principal symbol contains the light ray transform of the 1-form $A$. We use the fact that broken bicharacteristics converge uniformly to gliding rays on admissible Lorentzian manifolds. Since we have the light ray transform of $A$ along broken null-geodesics, the limit gives the integral of $A$ along the projection of gliding rays. We show that this is enough to construct $A$ on boundary tangential directions.
\end{enumerate}

Finally, when we already have a priori information, we show that more can be constructed. Instead of the recoverable sets, which by definition is the region connected by gliding rays, we show that we can construct the metric in the region connected by broken bicharacteristics. Note that this region is strictly larger than recoverable sets, {as gliding ray can be uniformly approximated by broken bicharacteristics on bounded time intervals}, see Section \ref{subsec: gliding}.

\subsection{Outline of the paper}
In Section \ref{sec: preliminaries}, we introduce tools and results required. Specifically, in Section \ref{subsec: operators on half-densities}, we introduce density bundles; in Section \ref{subsec: Lagrangian distributions}, we discuss microlocal analysis; in Section \ref{subsec: geometry}, we discuss the basics of Lorentzian manifold; in Section \ref{subsec: gauge equivalence}, we discuss gauge equivalence; and in Section \ref{subsec: gliding}, we discuss properties of broken bicharacteristics and gliding rays.

In Section \ref{sec: DN map}, we compute the principal symbol of $\Lambda^{U,V}_{g,A,q}$ for disjoint $U$ and $V$. In Section \ref{subsec: notation for dn map}, we introduce some notations; in Section \ref{subsec: optics}, we recall the optics construction of the solution; and in Section \ref{subsec: microlocal construction}, we symbolically construct the solution and compute the principal symbol.

In Section \ref{sec: recovery of weak lens relation}, we construct the conformal class of the boundary metric via weak lens relation; in Section \ref{sec: recover lens relation}, we upgrade the weak lens relation to lens relation up to time orientation near the glancing set; in Section \ref{sec: determine conformal factor}, we show the conformal factor is locally constant; in Section \ref{sec: determine 1-form}, we construct the 1-form $A$ on boundary tangential directions; finally in Section \ref{sec: proof of main theorems}, we combine the previous results and prove Theorem \ref{theorem: reconstruction}, Theorem \ref{theorem: determination} and Theorem \ref{theorem: reconstruction with a priori info}.

\subsection*{Acknowledgment} The author YY and YZ would like to thank Gunther Uhlmann for numerous helpful discussion throughout this project and to thank Plamen Stefanov for helpful suggestions.

\section{Preliminaries}\label{sec: preliminaries}

\subsection{Operators on half-densities}\label{subsec: operators on half-densities}


For $s \in \mathbb{R}$, the \textit{$s$-density bundle} is the real line bundle $\Omega_M^{\alpha} \to M$, whose transition map is defined as follows: for any coordinate charts $F_i: U_i \to \mathbb{R}^n$, the transition map is given by $\tau_{ij}(p, v) = (p, |\det(F_i \circ F_j^{-1})'|^{-s}v)$. Just like differential forms on a manifold, one can view $s$-densities as
\[
\Omega_{M}^s = \{\omega: \Lambda^nM \to \mathbb{R}: \omega(\mu v) = |\mu|^sv, \  v\in \Lambda^nM, \ \mu \in \mathbb{R}\},
\]
where $M$ has dimension $n$ and $\Lambda^nM$ is the $n$-th exterior power of tangent bundle.
In what follows, we will mainly use half-densities $\Omega_M^{1/2}$ and we write
\[
\omega(v_1, \dots, v_n) := \omega(v_1 \wedge \dots \wedge v_n).
\]
Hence if $w_1, \dots, w_n$ is another set of vectors and $A$ is the matrix satisfying $v_1 \wedge \dots \wedge v_n = (\det A) w_1 \wedge \dots \wedge w_n$, then
\[
\omega(v_1, \dots, v_n) = |\det A|^{1/2}\omega(w_1,\dots,w_n).
\]
We will also repeatedly use the notation $|dx|^{1/2}$ for the standard coordinate half-density, where $x$ is a local coordinate system, and $dx = dx^1 \wedge \dots \wedge dx^n$. For more details on densities, see \cite{hintz2025}.

Let $\mathcal{D}'(M;\Omega_M^{1/2})$ denote the set of half-density valued distributions on $M$, defined as the dual space of smooth half-densities $C^{\infty}(M; \Omega_M^{1/2})$. Let $|g|^{1/4} \in C^{\infty}(M;\Omega_M^{1/2})$ such that in local coordinates $x$, it has the form
\[
|g|^{1/4} = |g|_x^{1/4}|dx|^{1/2}
\]with $|g|_x$ being the absolute value determinant of $g$ in coordinate $x$. One can convert an operator on functions to one acting on half-densities by conjugating a half-density. The natural choice is given by conjugating $|g|^{1/4}$:
\[
P_{g,A,q} = |g|^{1/4}\Box_{g,A,q}|g|^{-1/4},
\]
so that $P_{g,A,q}(u|g|^{1/4}) = (\Box_{g,A,q}u)|g|^{1/4}$. In local coordinates, we have
\begin{align*}
|dx|^{-1/2}P_{g, A, q}(u(x)|dx|^{1/2}) &= |g|_x^{-1/4} (\partial_j - i A_j)
\left[ |g|_x^{1/2} g^{jk} (\partial_k - i A_k) (|g|_x^{-1/4}u) \right] + qu \\
&= \left[ g^{jk} \partial_j \partial_k + \partial_j g^{jk} \partial_k - 2i A_j g^{jk} \partial_k + {\text{lower order terms}} \right] u.
\end{align*}
Thus the principal symbol, sub-principal symbol and the corresponding Hamiltonian vector field for the operator $P_{g, A, q}$ are
\begin{align}
    p &= -g^{jk}\xi_j\xi_k,\nonumber\\
    c &= i\partial_jg^{jk}\xi_k + 2A_jg^{jk}\xi_k - \frac{1}{2i}\sum_l\frac{\partial^2}{\partial x^l \partial \xi_l} (-g^{jk}\xi_j\xi_k) = -A(H_p),\label{eq: symbol of P}\\
    H_p&= -2g^{jk}\xi_j\partial_{x^k} + \partial_{x^l}g^{jk}\xi_j\xi_k\partial_{\xi_l}\nonumber
\end{align}
where we lift $A$ as a 1-form on $T^*M$.

\subsection{Lagrangian distributions and principal symbols}\label{subsec: Lagrangian distributions}

In this subsection, we denote by $M$ a smooth manifold without boundary.
We follow \cite[Definition 3.2.2]{Hoermander1971}, with a small modification in notation.
\begin{defn}
    Let $\Lambda \subset T^*M\backslash 0$ be a conic Lagrangian submanifold. By $I^m(\Lambda; \Omega_M^{1/2})$ we shall denote the set of all $w \in \mathcal{D}'(M; \Omega_M^{1/2})$ such that $w = \sum_{j \in J} w_j$ with the supports of $w_j$ locally finite and
    \[
    \langle w_j, v \rangle = (2\pi)^{-(n+2N_j)/4}\int\int e^{i(\phi_j(x, \theta)-\pi N_j/4)}a_j(x, \theta)v(x)dxd\theta, \quad v\in C^{\infty}(M).
    \]
    In the above equation,
    \begin{itemize}
        \item $X_j$ is a local coordinate patch and $dx$ is the Lebesgue measure with respect to it;
        \item $\phi_j$ is a non-degenerate phase function defined in a conic neighborhood $U_j$ of $X_j \times (\mathbb{R}^{N_j}\backslash 0)$,  such that the critical set $\{(x, \theta): \partial_\theta \phi_j(x, \theta) = 0\}$  parameterizes $\Lambda$ locally via the diffeomorphism
        \[
        (x, \theta) \mapsto (x, \partial_x \phi_j(x, \theta));
        \]
        \item $a_j \in S^{\mu_j}(\mathbb{R}^n \times \mathbb{R}^{N_j})$ is a symbol of order $\mu_j = m+(n-2N_j)/4$,
        with $supp(a_j) \subset \{(x, t\theta): t\geq 1, (x, \theta) \in V \}$, where $V$ is a compact subset of the image $U_j$ in $\mathbb{R}^n\times \mathbb{R}^{N_j}$; see \cite[Definition 1.1.1]{Hoermander1971} for symbol classes.
    \end{itemize}
\end{defn}
In particular, if $\Lambda = N^*K$ is the conormal bundle of a smooth submanifold $K \subset M$, then $I^m(N^*K; \Omega_M^{1/2})$ is called the set of half-density valued conormal distributions.

It is well-known that given a Lagrangian distribution, its principal symbol can be invariantly defined as a half-density on the Lagrangian submanifold.
To define the principal symbol, consider $w \in I^m(\Lambda; \Omega_M^{1/2})$.
Suppose in a local coordinate patch, it has the form
\[
\int e^{i\phi(x, \theta)}a(x, \theta)d\theta |dx|^{1/2},
\]
where for simplicity we ignore the factor $(2\pi)^{-(n+2N)/4}$ throughout the paper, as it does not affect the computation at all.
Now let $C = \{(x, \theta): \partial_\theta \phi = 0\}$ be the critical set and $F(x, \theta) = (x, \phi_x)$ be the diffeomorphism from $C$ to an open subset of $\Lambda$.
Let $L$ be the Maslov bundle on $\Lambda$.
Then the principal symbol $\sigma(w) \in S^{m+n/4}(\Lambda; \Omega_{\Lambda}^{1/2} \otimes L)$ is locally computed via method of stationary phase, see for example \cite{Duistermaat2010, Hoermander1971}, as
\[
(F^{-1})^*\left( e^{i\pi N/4}a(x, \theta) \left|\frac{D(\lambda, \partial_{\theta}\phi)}{D(x, \theta)}\right|^{-1/2}|d\lambda|^{1/2} \right),
\]
where $\lambda$ is a local coordinates on $C$ and $N = sgn\left(\frac{D(\lambda, \partial_\theta\phi)}{D(x, \theta)}\right)$ is the Maslov index.
The principal symbol is unique modulo lower order symbols, i.e.,
\[
\sigma(w) \in S^{m+n/4}(\Lambda; \Omega^{1/2}_{\Lambda} \otimes L)/S^{m+n/4-1}(\Lambda; \Omega^{1/2}_{\Lambda} \otimes L)
\]
 is unique as an equivalence class.

We shall mostly work with classical symbols, i.e.,
$a(x, \theta)$ can be written asymptotically as
\[
    \chi(\theta) a(x, \theta) + (1-\chi(\theta))\sum_{j=m+(n-2N)/4}^{-\infty}a_j(x, \theta),
\]
where $\chi$ is a compactly supported smooth cutoff equal to $1$ around $0$, and $a_j$ is homogeneous of degree $j$ in $\theta$. In this case, the principal symbol is uniquely defined in $S^{m+n/4}(\Lambda; \Omega^{1/2}_{\Lambda} \otimes L)$, and we refer to such Lagrangian distributions as classical Lagrangian distributions, denoted by $I_{cl}^m(\Lambda;\Omega^{1/2}_{\Lambda})$.

We will repeatedly encounter Lagrangian distributions of the form $w = v|g|^{1/4}$ where $v \in \mathcal{D}'(M)$ and $|g|^{1/4} \in C^{\infty}(M;\Omega_M^{1/2})$ is the half-density defined previously. Then in local coordinates $v$ has the form
\[
v(x) = \int e^{i\phi(x, \theta)}b(x,\theta)d\theta = \int e^{i\phi(x, \theta)}|g|_x^{-1/4}a(x,\theta)d\theta.
\]
We shall denote
\[
\sigma_x(v) = (F^{-1})^*\left( e^{i\pi N/4}b(x, \theta) \left|\frac{D(\lambda, {\partial_\theta \phi})}{D(x, \theta)}\right|^{-1/2}|d\lambda|^{1/2} \right)
\]
with the subscript $x$ reminding that it depends on coordinate choice. It is easy to see that $\sigma(w) = \sigma_x(v)|g|^{1/4}_x$.

Finally, let $P$ be a differential operator acting on half-density valued distributions and $p(x,\xi)$ be its principal symbol.
Then we define the characteristic set $\Char(P) = \{(x, \xi)\in T^*M: p(x,\xi) = 0\}$.
Let $w \in I^m(\Lambda; \Omega_M^{1/2})$ be such that the Lagrangian submanifold $\Lambda$ is contained in the characteristic set of $P$, that is, $\Lambda \subset \Char(P)$. Suppose $Pw \in C^{\infty}(M; \Omega_M^{1/2})$.
Recall $H_p$ is the Hamiltonian vector field of $p$, and $c$ the sub-principal symbol of $P$.
Then the principal symbol $\sigma(w)$ satisfies
\begin{equation}\label{eq: lie derivative of principal symbol}
    (\mathcal{L}_{H_p}+ic)\sigma(w) = 0,
\end{equation}
where $\mathcal{L}_{H_p}$ is the Lie derivative of $H_p$ on half-densities over $\Lambda$, see \cite[Section 5.3]{DuiHor}.

For more details on the definition of Lagrangian distributions, principal symbols and FIO calculus, we refer to \cite{hintz2025, Dui, Hoermander1971, MR2304165, Hor4, Duistermaat1972}.

\subsection{Lorentzian geometry}\label{subsec: geometry}
Recall $(M,g)$ is an admissible Lorentzian manifold, see Definition \ref{def: admissible manifold}.
For $p, q \in M$, we denote by $p < q$ if $p \neq q$ and $p$ is joined to $q$ by a future-pointing causal curve.
We use $p \leq q$ if $p = q$ or $p < q$.
The causal future of $p$ is $J^+(p) =  \{q \in M: p \leq q\}$ and
that for any set $A \subseteq M$ is  $J^+(A) = \cap_{p \in A} J^+(p)$.
For $\eta \in T_p^*M$, the corresponding vector of $\eta$  is denoted by $ \eta^\sharp \in T_p M$.
The corresponding covector of  $v \in T_p M$ is denoted by $ v^\flat \in T^*_p M$.
We denote by
\[
L_p M = \{v \in T_p M \setminus 0: \  g(v, v) = 0\}
\]
the set of lightlike vectors at $p \in M$ and similarly by $L^*_p M$ the set of lightlike covectors.
The sets of future-pointing (or past-pointing) lightlike vectors are denoted by $L^+_p M$ (or $L^-_p M$), and those of future-pointing (or past-pointing) lightlike covectors are denoted by $L^{*,+}_p M$ (or $L^{*,-}_p M$).
Recall the principal symbol $p$ and the Hamiltonian vector field $H_p$ in (\ref{eq: symbol of P}) for $P_{g, A, q}$. 
Note that $P_{g,A,q}$ and $\Box_{g,A,q}$ share the same principal symbol and thus the same Hamiltonian vector field.
In the Lorentzian manifold $(M,g)$, the characteristic set $\Char(P_{g, A, q})$
is actually the set of lightlike covectors with respect to $g$.

The set $T^*\partial M \setminus 0$ is naturally decomposed into three different regions by the metric in the following way.
The elliptic region $\mathcal{E}$, the glancing region $\mathcal{G}$ and the hyperbolic region $\mathcal{H}$ are defined by
\begin{equation}\label{def: glancingset}
\begin{aligned}
    \mathcal{E} &= \{(x, \xi') \in T^*\partial M \setminus 0: |\pi^{-1}(x, \xi') \cap L^*M| = 0\},\\
    \mathcal{G} &= \{(x, \xi') \in T^*\partial M \setminus 0: |\pi^{-1}(x, \xi')\cap L^*M| = 1\},\\
    \mathcal{H} &= \{(x, \xi') \in T^*\partial M \setminus 0: |\pi^{-1}(x, \xi')\cap L^*M| = 2\}.
\end{aligned}
\end{equation}
If $(x, \xi') \in \mathcal{H}$, we thus have two lightlike preimages $(x, \xi_{\pm})$, where $+$/$-$ corresponds to $\xi_{\pm}^{\sharp}$ being inward/outward-pointing.
Let $\Gamma_{x, \xi_{\pm}}$ be the integral curves generated by $H_p$ passing through $(x, \xi_{\pm})$.
We call them the (null-)bicharacteristics for $P_{g,A,q}$ or $\Box_{g,A,q}$.
Let $\gamma_{x,\xi_{\pm}}$ be the projection of $\Gamma_{x, \xi_{\pm}}$ onto $M$.
Note that $\gamma_{x, \xi_{\pm}}$ is simply the unique null-geodesic passing through $x$ with direction $\dot{\gamma}_{x, \xi_{\pm}} = -2\xi_{\pm}^{\sharp}$.
In the semi-geodesic normal coordinates, the unique inward/outward-pointing preimages of $(x', \xi')$ are $(x, \xi_\pm)$, where we write $x = (x', 0)$ and $\xi_\pm$ are given by
\[
\xi_\pm = (\xi', \pm \xi_n), \quad \text{ with } \xi_n := \sqrt{-g^{\alpha \beta} \xi_{\alpha}\xi_{\beta}}.
\]
The $\partial_{x^n}$ component of the Hamiltonian vector field at $(x', 0, \xi_{\pm})$ is thus $\mp 2\xi_n \partial_{x^n}$.
This implies $\Gamma_{x, \xi_+}$ travels in the opposite direction as the parameter increases, that is,  $\Gamma_{x, \xi_+}((-s, 0)) \subset M^{\circ}$ for some $s > 0$.
On the other hand, $\Gamma_{x, \xi_-}$ travels in the same direction as the parameter increases, that is, $\Gamma_{x, \xi_-}((0, s)) \subset M^{\circ}$ for some $s > 0$.
Roughly speaking, when $s$ increases, a bicharacteristic $\Gamma(s)$ starting from the boundary transversally will always leave from some outward-pointing $(x, \xi_-)$ and then hits the boundary transversally at some inward-pointing $(y, \eta_+)$.

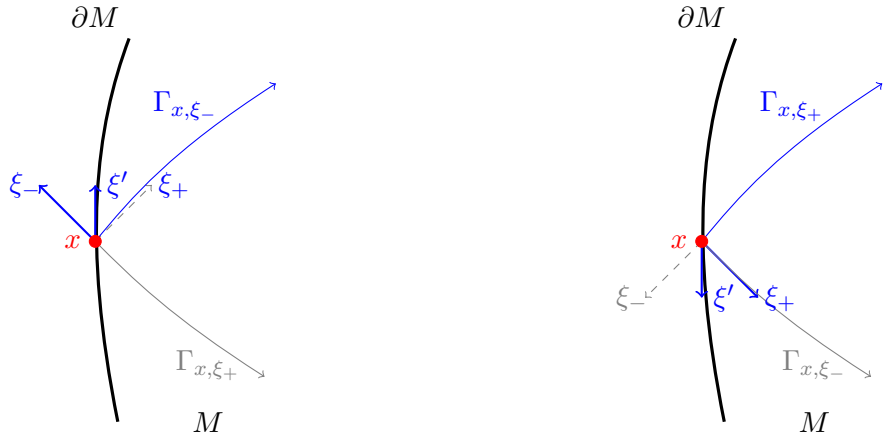
\begin{figure}[h]
        \begin{minipage}{0.48\textwidth}
        \centering
        \begin{tikzpicture}[scale=1.5]
            
            \node at (1,-0.6) {\small $M$};
            \draw[very thick] (0.2,-0.6) .. controls (-0.12,1) and (0,2) .. (0.3,2.8) (0,2.8) node[above] {\small $\partial M$};
            
            
                \draw[gray, ->, dashed] (0, 1) -- (0.5,  1.5);
                \draw[blue, thick, ->] (0, 1) -- (-0.5,  1.5);
                \draw[blue, thick, ->] (0, 1) -- (0,  1.5);

            \draw[gray, ->] (0, 1) .. controls (0.33,0.67) and (0.67,0.33) .. (1.5, -0.2);
            \node[gray] at (1,-0.1) {\small $ \Gamma_{x,\xi_{+}}$};    
            
            \draw[blue, ->] (0, 1) .. controls (0.33,1.4) and (0.67,1.8) .. (1.6, 2.4);
            \node[blue] at (0.8,2.2) {\small $ \Gamma_{x,\xi_{-}}$};    
            
                \node[blue] at (0.2,1.5) {\small $ \xi'$};   
                \node[blue] at (0.7,1.5) {\small $ \xi_{+}$};   
                \node[blue] at (-0.62,1.5) {\small $ \xi_{-}$};

            \foreach \x/\y in {0/1} {\filldraw[red] (\x,\y) circle (1.5pt);}
            \node[red] at (-0.2,1) {\small $x$};   
            
        \end{tikzpicture}
    \end{minipage}
    \begin{minipage}{0.48\textwidth}
        \centering
        \begin{tikzpicture}[scale=1.5]
            \node at (1,-0.6) {\small $M$};
          \draw[very thick] (0.2,-0.6) .. controls (-0.12,1) and (0,2) .. (0.3,2.8) (0,2.8) node[above] {\small $\partial M$};
          
          \draw[blue, thick, ->] (0, 1) -- (0.5, 0.5);
          \draw[gray, ->, dashed] (0, 1) -- (-0.5, 0.5);
          \draw[blue, thick, ->] (0, 1) -- (0, 0.5);
          
          \draw[gray, ->] (0, 1) .. controls (0.33,0.67) and (0.67,0.33) .. (1.5, -0.2);
          \node[gray] at (1,-0.1) {\small $ \Gamma_{x,\xi_{-}}$};    
          
          \draw[blue, ->] (0, 1) .. controls (0.33,1.4) and (0.67,1.8) .. (1.6, 2.4);
          \node[blue] at (0.8,2.2) {\small $ \Gamma_{x,\xi_{+}}$};    
        
          \node[blue] at (0.2,0.5) {\small $\xi'$};
          \node[blue] at (0.7,0.5) {\small $ \xi_{+}$};
          
          \node[gray] at (-0.62,0.5) {\small $ \xi_{-}$}; 
          
          \foreach \x/\y in {0/1} {\filldraw[red] (\x,\y) circle (1.5pt);}
          \node[red] at (-0.2,1) {\small $x$};   
          
        \end{tikzpicture}
    \end{minipage}
    \caption{If $(x', \xi')$ is past-pointing (left), then the unique forward null-bicharacteristic leaves from its outward preimage $(x, \xi_-)$ and travels forward in time; if $(x', \xi')$ is future-pointing (right), then the unique forward null-bicharacteristic leaves from its inward preimage $(x, \xi_+)$ and travels backward in time. In conclusion, the null-bicharacteristic is forward when it lies in the casual future of $x$, regardless of the flow direction.}
    \centering
\end{figure}

If the time orientation is given, then a vector $v$ is called future-pointing if it is in the future-pointing light cone.
A covector $\xi$ is called future-pointing if the corresponding vector $\xi^{\sharp}$ is.
Recall we assume solutions to the wave equation are only supported in the causal future of the initial data.
In other words, information travels along the null-bicharacteristics into the causal future.
For every boundary covector $(x', \xi') \in \mathcal{H}$,
its preimages $(x, \xi_{\pm})$ give us four choices of the null-bicharacteristic, depending on whether it is contained in the interior or exterior of $M$ and whether it is in the causal future or causal past of $x$.
Here recall we write $x = (x',0)$.
Thus, there exists a unique null-bicharacteristic starting from either $(x, \xi_+)$ or $(x, \xi_-)$, and lying in the interior of $M$ as well as the causal future of $x'$.
We call it the unique \textit{forward null-bicharacteristic} corresponding to $(x', \xi')$.
When $(x', \xi')$ is future-pointing, the forward null-bicharacteristic is the one traveling in the opposite direction as the parameter increases, starting from $(x, \xi_+)$; when $(x', \xi')$ is past-pointing, the forward null-bicharacteristic is the one traveling in the opposite direction as the parameter increases, starting from $(x, \xi_-)$.
We say $\gamma$ is the unique \textit{forward null-geodesic} corresponding to $(x', \xi')$, if it is the projection of the unique forward null-bicharacteristic.
Note that the natural parameterization of $\gamma$ would give us $\dot{\gamma}(0) = -2(\xi')^{\sharp}$ by \eqref{eq: symbol of P}.
We remark that our definition of the forward null-bicharacteristic corresponds to the future-pointing null-bicharacteristic defined in \cite{Stefanov2018}, but it flows in the opposite direction.
This is because we chose $p = -g^{jk}\xi_j\xi_k$, whereas their principal symbol is $p = \frac{1}{2}g^{jk}\xi_j\xi_k$.

Note that under the strictly null-convex assumption, if a null-geodesic starting from the interior hits the boundary, then it must do so transversally,
see \cite[Proposition 2.4]{Hintz2017}.
This implies that null-bicharacteristics starting from the interior only intersect the boundary in the hyperbolic set.
In addition, we state the following proposition, using the ideas from \cite{SaBWun} and \cite[Proposition 7.1]{SaBUhlWan}.
\begin{prop}\label{prop: Lambda_restriciton}
    Suppose $\pM$ is timelike and strictly null convex.
    Let $\Lambda \subset T^*\intM \setminus 0$ be a smooth conic Lagrangian submanifold contained in the characteristic set $\Char(P_{g, A, q})$.
    Suppose $\Lambda$ intersects $\partial M$.
    Then the intersection is transversal and therefore
    \[
    \Lambda' \coloneqq \Lambda|_{\pM},
    \]
    is a smooth conic Lagrangian submanifold of $T^*\pM$.
    We may extend $\Lambda$ across $\pM$ as a smooth conic Lagrangian submanifold of a larger Lorentzian manifold $\tM$.
    Moreover, if $u \in I^{\mu + \frac{1}{4}}(\Lambda)$ is a Lagrangian distribution,
    then its restriction
    \[
    u_0 \in I^\mu(\Lambda')
    \]
    is also a Lagrangian distribution.
\end{prop}
\begin{proof}
    By \cite[Proposition 2.4]{Hintz2017}, the Hamiltonian vector field $H_p$ is transversal to $\pM$, as its projection to $M$, i.e., the null vector field, hits $\pM$ transversally.
    Since $\Lambda \subset \Char(P_{g, A, q})$, one has $H_p$ is tangent to $\Lambda$ and therefore $\Lambda$ intersects $\pM$ transversally.
    As observed in \cite{SaBWun}, $\Lambda$ can be extended across the boundary as the union of integral curves of $H_p$, starting from $\Lambda$ in $\intM$.
    Then the same proof in \cite[Proposition 7.1]{SaBUhlWan} indicates the rest of the statements.
\end{proof}

Now we define the lens relation.
\begin{defn}\label{def: lens}
    Let $(x', \xi') \in \mathcal{H}$, and let $\Gamma$ be the unique forward bicharacteristic corresponding to $(x', \xi')$. Denote by $\mathcal{H}'$ the set of $(x', \xi') \in \mathcal{H}$ such that $\Gamma$ hits the boundary again at some $(y, \eta)$. Then the forward lens relation $L$ is defined by
    \[
    L: \mathcal{H}' \to \mathcal{H}, \quad (x', \xi') \to (y', \eta'),
    \]
    where $(y', \eta') = \pi(y, \eta)$ with the natural projection $\pi: T^*M|_{\partial M} \to T^*\partial M$. In particular, $L$ is invertible on its range, and $L^{-1}$ is the backward lens relation. For all $k \in \mathbb{Z}$, $L^k$ is called the $k$-th lens relation.
\end{defn}
\begin{rem}\label{rem: lens relation}
We emphasize that $L$ can be defined locally near $(x', \xi')$.
Moreover, $L$ is called the forward lens relation, as $y'$ is always in the causal future of $x'$.
Throughout this work, we always use $L$ to denote the forward lens relation, although for convenience we shall refer to it simply as the lens relation.
\end{rem}


\subsection{Gauge equivalence}\label{subsec: gauge equivalence}
There exist some natural gauge transformations such that the restricted DN map $\Lambda^{U,V}_{g,A,q}$ is invariant.
In the following, we first introduce the semi-geodesic normal coordinates on Lorentzian manifold with timelike boundary.
\begin{lem}[\cite{Stefanov2018} Lemma 2.3]\label{lem: semi-geodesic}
    For every $x \in \partial M$, there exists $\epsilon > 0$, a neighborhood $N$ of $x$ in $M$, and a diffeomorphism $\Psi: \partial M \cap N \times [0, T) \to N$ such that
    \begin{enumerate}
        \item $\Psi(x', 0) = x'$ for all $x' \in \partial M \cap N$;
        \item $\Psi(x', x^n) = \gamma_{x'}(x^n)$ where $\gamma_{x'}(x^n)$ is the unit speed geodesic issued from $x'$ normal to $\partial M$.
    \end{enumerate}
    If we write $x' =(x^1, \dots, x^{n-1})$ as local boundary coordinates on $\partial M$, in the coordinate system $(x', x^n)$, the metric $g$ takes the form
    \[
    g = g_{\alpha\beta}dx^{\alpha}\otimes dx^{\beta}+dx^n \otimes dx^n, \quad \alpha,\beta \leq n-1.
    \]
    For any two metrics $g_1$ and $g_2$, there exists a boundary fixing diffeomorphism $\Psi: M \to M$, $\Psi|_{\partial M} = \mathrm{Id}_{\partial M}$, such that $(x', x^n)$ is the semi-geodesic normal coordinates for $g_1$ if and only if it is the semi-geodesic normal coordinates for $\Psi^* g_2$.
\end{lem}

Now we state the gauge equivalence, following the same argument as in \cite{Stefanov2018}.
\begin{lem}
    Let $(M, g)$ be an $n$-dimensional admissible Lorentzian manifold with boundary.
    Let $A$ be a smooth 1-form and $q$ be a smooth function on $M$.
    Suppose $U$ and $V$ are open connected subsets of $\pM$ such that $U \cap V = \emptyset$.
    Suppose $\Psi: M \to M$ is a diffeomorphism such that $\Psi|_{U \cup V} = \mathrm{Id}|_{U\cup V}$, then
    \[
    \Lambda^{U,V}_{g,A,q} = \Lambda^{U,V}_{\Psi^*g,\Psi^*A,\Psi^*q}.
    \]
\end{lem}
\begin{proof}
    The proof follows directly from \cite[Lemma 2.1]{Stefanov2018}.
\end{proof}

Another class of gauge transformations is given by conformal changes, but unlike the case of full boundary data (see \cite[Lemma 2.2]{Stefanov2018}), we have fewer restrictions here.
\begin{lem}\label{lem: gauge2}
    Let $(M, g)$ be an $n$-dimensional admissible Lorentzian manifold with boundary.
    Let $A$ be a smooth 1-form and $q$ be a smooth function on $M$. Suppose $U$ and $V$ are open connected subsets of $\pM$ such that $U \cap V = \emptyset$.
    Let $\varphi$ and $\psi$ be smooth functions satisfying
    \[
    \varphi|_{U} \equiv \frac{c}{n-2}, \quad  \varphi|_{V} \equiv \frac{c}{n}, \quad \psi|_{U \cup V} = 0,
    \]
    for some constant $c$. Then
    \[
    \Lambda^{U, V}_{g,A,q} = \Lambda^{U,V}_{e^{-2\varphi}g, A-d\psi,e^{2\varphi}(q + q_{\varphi})}
    \]
    where $q_{\varphi} = e^{\frac{n-2}{2}\varphi}\Box_g e^{\frac{2-n}{2}\varphi}$.
\end{lem}
\begin{proof}
    Same computation as in \cite[Lemma 2.2]{Stefanov2018} shows that if $v = e^{-c}e^{\frac{n-2}{2}\varphi}e^{-i\psi}u$ and $u, f$ satisfies $\Box_{g,A,q}u = 0$, $u|_{\partial M} = f$, then
    \[
    \Box_{e^{-2\varphi}g, A-d\psi,e^{2\varphi}(q+q_{\varphi})}v = 0, \quad v|_{\partial M} = f.
    \]
    By our assumptions on $\varphi$, we have $\nu_{e^{-2\varphi}g} = e^{\varphi}\nu_g = e^{\frac{2}{n}c}\nu_g$ on $V$.
    Use the fact that $U\cap V = \emptyset$, which means $\Lambda^{U,V}_{g,A,q}f = \partial_{\nu}u|_V$, we have
    \begin{equation*}
        \Lambda^{U,V}_{e^{-2\varphi}g, A-d\psi,e^{2\varphi}(q {\blue +} q_{\varphi})}f = e^{\frac{2}{n}c}\partial_{\nu}v|_V = e^{\frac{2}{n}c-c+\frac{n-2}{n}c}\partial_{\nu}u|_V = \Lambda^{U,V}_{g,A,q}f.
    \end{equation*}
\end{proof}

\subsection{Broken bicharacteristics and gliding rays}\label{subsec: gliding}


We first define the broken null-geodesics (see \cite[Definition 2.7]{Hintz2017}) and the broken bicharacteristics (see \cite[Definition 24.2.2]{MR2304165}).
\begin{defn}\label{def: brokengeo}
    Let $I \subset \mathbb{R}$ be an open connected interval.
    A piecewise smooth curve $\gamma: I \to M$ is called a \emph{broken null-geodesic} if
    \begin{enumerate}
        \item for all open intervals $J \subset I$ with $\gamma(J) \cap \partial M = \emptyset$, $\gamma|_J$ is an affinely parametrized null-geodesic in $(M^{\circ}, g)$;
        \item if $s \in I$ and $\gamma(s) \in \partial M$, then $\gamma|_{(s-\epsilon, s]}$ and $\gamma|_{[s, s+\epsilon)}$ are null-geodesics with $\gamma(s\pm (0, \epsilon)) \subset M^{\circ}$ for small $\epsilon > 0$, and we have $\pi(\dot{\gamma}(s+0))^{\flat} = \pi(\dot{\gamma}(s-0))^{\flat}$.
    \end{enumerate}
\end{defn}
\begin{defn}\label{def: brokenbichar}
    Let $I \subset \mathbb{R}$ be an open connected interval and $B \subset \mR$ be a discrete subset.
    A \emph{broken bicharacteristic} of $P_{g,A,q}$ is a map $\Gamma: I\backslash B \to T^*M^{\circ}$  such that
    \begin{enumerate}
        \item if $J \subset I\backslash B$ is an interval, then $\Gamma|_J$ is a bicharacteristic of $P_{g,A,q}$ in $M^{\circ}$;
        \item if $s \in B$,  then $\Gamma(s\pm 0)$ exist and belong to $T^*M\backslash 0|_{\partial M}$ with $\pi(\Gamma(s+0)) = \pi(\Gamma(s-0))$.
    \end{enumerate}
\end{defn}
For notation simplicity, when $s \in B$, we define $\Gamma(s) = \pi(\Gamma(s+0)) = \pi(\Gamma(s-0)) \in T^*\partial M$.
Thus, the broken bicharacteristic is defined on the entire $I$ and takes value in $T^*M^{\circ} \sqcup T^*\partial M$. We shall use $\Gamma|_{\partial M} \subset T^*\partial M$ to denote these boundary reflection points.
Note that broken bicharacteristic is not continuous, since when it reaches the boundary, it needs to switch from $(x, \xi_\pm)$ to $(x, \xi_\mp)$. A common way of fixing the non-continuous property is to define the compressed bundle and compressed broken bicharacteristic, which we do not go into details here, see \cite[Definition 24.3.7]{MR2304165}.

It is easy to see that the projection of a broken bicharacteristic is a broken null-geodesic.
By \cite[Proposition 2.12]{Hintz2017}, all inextendible broken null-geodesics are tame. Recall the following definition of tameness.
\begin{defn}[{\cite[Definition 2.11]{Hintz2017}}]
    Let $a \geq -\infty$ and $b \leq \infty$.
    We say an inextendible broken null-geodesic $\gamma: I \to \mathbb{R}$ is tame for $a < t < b$, if for all $a< a'$, $b<b'$, one has
    \[
    t(\gamma(I)) \cap (a, a')\neq \emptyset, \quad t(\gamma(I)) \cap (b', b) \neq \emptyset.
    \] If $\gamma$ is tame for $-\infty < t < \infty$, we simply say $\gamma$ is \emph{tame}.
\end{defn}
\begin{prop}[\cite{Hintz2017} Proposition 2.12]\label{prop: broken null-geod are tame}
    Let $(M, g)$ be an $n$-dimensional admissible Lorentzian manifold with boundary.
    Let $\gamma: I \to M$ be an inextendible broken null-geodesic. Then $I = \mathbb{R}$, and $\gamma$ is tame.
\end{prop}

Next we define gliding rays. Suppose $\phi$ is a boundary defining function such that $\phi > 0$ in $M^{\circ}$.
We introduce the following terminology for the glancing set, see \cite[Definition 24.3.2]{MR2304165}.
\begin{defn}
    The glancing set $\mathcal{G}^k$ of order at least $k \geq 2$ is defined by
    \[
    \{p = 0 \text{ and } H_p^j\phi = 0 \text{ for } 0\leq j < k\}.
    \]
   Note that $\mathcal{G} = \mathcal{G}^2 \supset \mathcal{G}^3 \supset \cdots \supset \mathcal{G}^{\infty}$. The glancing set $\mathcal{G}^2\backslash \mathcal{G}^3$ of order precisely 2 is the union $\mathcal{G}_d \cup \mathcal{G}_g$ of the diffractive part $\mathcal{G}_d$ where $H_p^2\phi > 0$ and the gliding part $\mathcal{G}_g$ where $H_p^2\phi < 0$.
\end{defn}
The definition is invariant on the choice of $\phi$.
In our case, all glancing covectors are in fact gliding, that is $\mathcal{G} = \mathcal{G}_g$.
\begin{lem}\label{lemma: glancing is gliding}
    If $\partial M$ is strictly null-convex, then $\mathcal{G} = \mathcal{G}_g$.
\end{lem}
\begin{proof}
    One can explicitly compute $H^2_p\phi$ in local coordinate:
    \begin{align*}
        \frac{1}{4}H^2_p\phi &= \frac{1}{4}(-2g^{jk}\xi_j\partial_k+\partial_ig^{jk}\xi_j\xi_k\partial_{\xi_i})(-2g^{lm}\xi_l\partial_m\phi)\\
        &= g^{jk}g^{lm}\xi_j\xi_l\partial^2_{km}\phi + g^{jk}\partial_kg^{lm}\xi_j\xi_l\partial_m\phi - \frac{1}{2}g^{lm}\partial_lg^{jk}\xi_j\xi_k\partial_m\phi.
    \end{align*}
    On the other hand, by some scaling, we can assume $\nabla \phi = -\nu$ is the inward unit normal vector, then
    \begin{equation*}
        -g(\nabla_{\xi^{\sharp}}\nu, \xi^{\sharp})= g^{jk}\xi_j\xi_l\partial_m\phi(\partial_kg^{lm}+\Gamma^l_{ik}) + g^{jk}g^{lm}\xi_j\xi_l\partial^2_{km}\phi,
    \end{equation*}
    and the first term becomes
    \begin{align*}
        &\quad g^{jk}\xi_j\xi_l\partial_m\phi(\partial_kg^{lm} - \frac{1}{2}g^{im}g_{kp}\partial_ig^{lp} - \frac{1}{2}g^{im}g_{ip}\partial_kg^{lp}+\frac{1}{2}g_{ik}g^{lp}\partial_pg^{im})\\
        &=\xi_l\xi_j\partial_m\phi(g^{jk}\partial_kg^{lm} - \frac{1}{2}g^{im}\partial_ig^{jl}-\frac{1}{2}g^{jk}\partial_kg^{lm} + \frac{1}{2}g^{lp}\partial_pg^{jm})\\
        &= \xi_l\xi_j\partial_m\phi(g^{jk}\partial_kg^{lm}-\frac{1}{2}g^{im}\partial_ig^{jl}).
    \end{align*}
    Hence by switching indices, and using the strict null-convexity, we have
    \[
    H^2_p \phi = -4g(\nabla_{\xi^{\sharp}} \nu, \xi^{\sharp}) < 0.
    \]
\end{proof}
Since the glancing set coincides with the gliding set, we define the gliding vector field and gliding rays as follows, see \cite[Definition 24.3.6]{MR2304165} and discussions there for more details.
\begin{defn}
    Let $\phi$ be a boundary defining function. The vector field
    \[
    H_p^G = H_p + \frac{H_p^2\phi}{H_{\phi}^2p}H_{\phi}
    \]
    tangent to $\mathcal{G}$ is called the \emph{gliding vector field}. The gliding ray is the integral curve generated by the gliding vector field.
\end{defn}
Again the definition is invariant on the choice of $\phi$. Similar to bicharacteristics and null-geodesics, we say a gliding ray is forward (resp. backward) with respect to a glancing direction if it lies in the causal future (resp. causal past) of the base point.

We next state a uniform convergence result relating broken bicharacteristics and gliding rays. First we explain what we mean by uniform convergence. The metric $g$ canonically identifies $T^*\partial M$ with the subspace
\[
(T\partial M)^{\flat} = \{v^{\flat}: v \in T\partial M\} \subset T^*M|_{\partial M}
\]
 via the map
\[
T^*\partial M \ni \xi' \mapsto ((\xi')^{\sharp})^{\flat},
\]
where the first $\sharp$ is taken with respect to $g|_{T\partial M \times T\partial M}$ so that $(\xi')^{\sharp} \in T\partial M$, and the second $\flat$ is with respect to $g$.
We will henceforth use this identification without further mention. In semi-geodesic normal coordinates, this map is simply
\[
(x', \xi') \mapsto (x', 0, \xi', 0).
\]
A gliding ray can thus be viewed as a path in $T^*M|_{\partial M}$. The uniform convergence we refer to is the uniform convergence in $T^*M$ with respect to some arbitrarily chosen Riemannian metric on $T^*M$.
In the following, let $\Gamma^b_{y', \eta'}$ denote the broken null-bicharacteristic corresponding to $(y', \eta') \in \mathcal{H}$ and let $\Gamma^{gl}_{x', \xi'}$ denote the gliding ray corresponding to $(x', \xi') \in \mathcal{G}$.
\begin{lem}\label{lemma: uniform convergence of broken bichar}
    Let $(x_j', {\xi^j}') \in \mathcal{H}$ be a sequence converging to $(x', \xi') \in \mathcal{G}$.
    Let $\Gamma^b_j(t)$ be the broken null-bicharacteristic such that $\Gamma^b_j(0) = (x'_j, {\xi^j}')$ and let $\Gamma^{gl}(t)$ be the gliding ray 
    such that $\Gamma^{gl}(0) = (x', \xi')$.
    Then for any fixed $T>0$,
    the broken null-bicharacteristic segment $\Gamma^b_j([-T, T])$ converges uniformly to the gliding ray segment $\Gamma^{gl}([-T, T])$, with respect to any fixed Riemannian metric on $T^*M$.
    In particular, any inextendible gliding ray can be defined on the entire $\mathbb{R}$.
\end{lem}
\begin{proof}
    See Lemma 24.3.5 of \cite{MR2304165} and the discussion after it.
\end{proof}

Recall the restricted metric $\bar{g} = g|_{T\partial M \times T \partial M}$ is Lorentzian on the boundary. We prove that the projection of gliding rays on the boundary are precisely the null-geodesics given by the boundary metric $\bar{g}$, and they are also tame.
\begin{lem}\label{lemma: gliding rays are tame}
     Let $(M, g)$ be an $n$-dimensional admissible Lorentzian manifold with boundary. Let $\gamma^{gl}:I \to M$ be the projection of some inextendible gliding ray $\Gamma^{gl}$.
     Then $\gamma^{gl}$ is a null geodesic on $\pM$ with respect to $\bar{g}$, and moreover
     it is tame.
\end{lem}
\begin{proof}
    Use the identification of $T^*\partial M \subset T^*M|_{\partial M}$, the gliding vector field is the Hamiltonian vector field of $\tilde{p} = p|_{T^*\partial M}$. On the other hand, in the semi-geodesic normal coordinates, $\tilde{p} = -g^{\alpha \beta}\xi_{\alpha}\xi_{\beta} = -\bar{g}^{\alpha \beta}\xi_{\alpha}\xi_{\beta}$. Thus, gliding rays are the null-bicharacteristics with respect to $\tilde{p}$, and their projections are the null-geodesics in $(\partial M, \bar{g})$. An inextendible gliding ray corresponds to an inextendible boundary null-geodesic $\gamma^{gl}: I \to \partial M$. Clearly $\gamma^{gl}(s)$ has no limiting point as $s \to \sup I$, since otherwise can extend out as a null-geodesic (same for $s \to \inf I$).
    Since $M$ has a temporal function, by \cite[Proposition 6.4.9]{hawking2023large} the Lorentzian submanifold $(\pM,\bar{g})$ is stably causal, see also \cite{minguzzi2010time}.
    Since a stably causal spacetime must be strongly causal, then any inextendible causal curve in $(\pM,\bar{g})$ can not be imprisoned in a compact set, according to \cite[Proposition 3.13]{Beem2017}.
    Using the assumption that the temporal function is proper, we can prove the boundary null-geodesics are tame.
\end{proof}

\begin{exmp}
    Consider the simple case $g = -f(t)dt^2 + dx^2$ on an infinite unit cylinder, where $f > 0$. Let $\phi = 1 - |x|^2$ be the boundary defining function. Then the gliding vector field is
    \[
    H^G_p = 2f(t)\tau \partial_t - f'(t)\tau^2 \partial_{\tau} + a \cdot \partial_x + b \cdot \partial_{\xi}.
    \]
    The time coordinate of gliding ray is thus given by
    \[
    \begin{cases}
        \dot{t} = 2f(t) \tau\\
        \dot{\tau} = -f'(t)\tau^2
    \end{cases}
    \implies \ddot{t} = \frac{1}{2}\frac{f'(t)}{f(t)}\dot{t}^2.
    \]
    A direct computation shows for some $C > 0$,
    \[
    \dot{t}(s) = C\sqrt{f(t(s))}.
    \]
    If the gliding ray is trapped on one end, say, as $s \to \infty$, then $t(s)$ would be bounded and close to some constant for all large $s$. Then $\dot{t}(s)$ must be bounded below by a positive constant for all large $s$. This contradicts with $t(s)$ being bounded, and therefore the gliding rays escape any compact region.
\end{exmp}

Since we can only observe the broken bicharacteristic when it hits the boundary, we need to show that the uniform convergence in Lemma \ref{lemma: uniform convergence of broken bichar} can be detected with boundary information. That is, because of strict null-convexity, the boundary points along the broken bicharacteristics can approximate the gliding ray.
\begin{lem}\label{lemma: broken bichar converge to gliding on boundary}
    Let $(M, g)$ be an $n$-dimensional admissible Lorentzian manifold with boundary.
    Suppose $(x_j', {\xi^j}') \in \mathcal{H}$ is a sequence converging to $(x', \xi') \in \mathcal{G}$.
    Let $\Gamma^b_j(t)$ be the broken null-bicharacteristic such that $\Gamma^b_j(0) = (x'_j, {\xi^j}')$.
    Let $I \subset \mR$ be an open connected interval containing $0$.
    Recall we denote by $\Gamma^b_j(I)|_{\partial M} \subset T^*\partial M$ the set of boundary points.
    Let $\Gamma^{gl}(t)$ be the gliding ray 
    such that $\Gamma^{gl}(0) = (x', \xi')$.
    Then
    \begin{enumerate}
        \item for any fixed $t \in I$, there exists $s_j \in I$ such that $s_j \to t$, $\Gamma^b_j(s_j) \in \Gamma^b_j(I)|_{\partial M}$, and $\Gamma^b_j(s_j) \to \Gamma^{gl}(t)$;
        \item if $(y', \eta')$ is a limiting point of $(y'_j, {\eta^j}') \in \Gamma^b_j(\mathbb{R})|_{\partial M}$, then $(y', \eta') = \Gamma^{gl}(t)$ for some $t \in \mathbb{R}$. That is, the boundary points of this sequence of broken null-bicharacteristics can only converge to some point in $\Gamma^{gl}$.
    \end{enumerate}
\end{lem}
\begin{proof}
    We first prove (1).
    Note that we can assume $I$ is bounded, since it is a local result.
    We will show that when $j \to \infty$, the maximum time $\Gamma^b_j$ spent in the interior between two consecutive reflections goes to 0 on $I$. Let $s' \in I$ be some time such that $\Gamma^b_j(s') \in T^*\partial M$, which is the same as $\gamma^b_j(s') \in \partial M$, where $\gamma^b_j$ is the corresponding broken null-geodesic. We can always extend the manifold to be $M_e$ such that $\partial M$ is now in the interior by \cite[Lemma 2.2]{Hintz2017}. Let $\gamma$ be the null-geodesic in $M_e$ such that $(\gamma(0), \dot{\gamma}(0)^{\flat}) = \Gamma^{gl}(s')$. In semi-geodesic normal coordinates, by strict null-convexity we have
    \[
    \left.\frac{d}{ds^2}\right|_{s=0}(x^n(\gamma(s))) = -g(\nabla_{\dot{\gamma}(0)} \nu, \dot{\gamma}(0)) < 0.
    \]
    Since $\dot{\gamma}(0)$ is tangential to the boundary, we have that $\gamma((0, \e)) \subset M_e \backslash M$ for arbitrarily small $\e$. For any such $\e$, let $W_{\e} \subset M_e\backslash M$ be a neighborhood of $\gamma(\e)$, then for any null-geodesic $\tilde{\gamma}$ such that $(\tilde{\gamma}(0), \dot{\tilde{\gamma}}(0))$ sufficiently close to $(\gamma(0), \dot{\gamma}(0))$, $\tilde{\gamma}(\e) \in W_{\e}$. In other words, $\tilde{\gamma}$ intersects with $\partial M$ in time smaller than $\e$. If $j$ is sufficiently large, by Lemma \ref{lemma: uniform convergence of broken bichar}, $\Gamma^b_j(s')$ would be sufficiently close to $\Gamma^{gl}(s')$, meaning $\gamma^b_j(s'+\delta) \in \partial M$ for some $\delta < \e$. Since the convergence from $\Gamma^b_j(I)$ to $\Gamma^{gl}(I)$ is uniform and $I$ is compact, we obtain that for any $\e > 0$, there exists $J$ such that for all $j \geq J$, $\Gamma^b_j(s') \in T^*\partial M$ implies $\Gamma^b_j(s'') \in T^*\partial M$ for some $|s' - s''| < \e$.

    Now let $t \in I^{\circ}$, $\delta > 0$. There exists $\e > 0$ such that under the fixed Riemannian metric on $T^*M$, $\dist(\Gamma^{gl}(t), \Gamma^{gl}(s)) < \e/2$ for any $|s - t| < \delta$. By previous argument, pick $J$ sufficiently large, for all $j \geq J$, there exists some $s_j$ such that $\Gamma^b_j(s_j) \in T^*\partial M$ and $|s_j - t| < \delta$ (can assume $\delta$ sufficiently small with respect to chosen $t$ such that $(t-\delta, t+\delta) \subset I$). By Lemma \ref{lemma: uniform convergence of broken bichar}, can let $J$ be sufficiently large such that the distance between $\Gamma^b_j$ and $\Gamma^{gl}$ is less than $\e/2$ on $I$. Then
    \[
    \dist(\Gamma^{gl}(t), \Gamma^b_j(s_j)) \leq \dist(\Gamma^{gl}(t), \Gamma^{gl}(s_j)) + \dist(\Gamma^{gl}(s_j), \Gamma^b_j(s_j)) < \e.
    \]

    Now we prove (2). Suppose $(y', \eta')$ is a limiting point of a sequence of $(y'_j, \eta'_j) \in B^{\mathbb{R}}_j$, but it is not in $\Gamma^{gl}$. Consider some smooth, surjective, proper temporal function $\tau$, then there exists some bounded time interval $[a, b]$ such that $y'_j, y' \in \tau^{-1}([a, b])$. There exists sufficiently large $N$ such that $\tau(\Gamma^{gl}(N)) > b$ and $\tau(\Gamma^{gl}(-N)) < a$ by Lemma \ref{lemma: gliding rays are tame}. Then for sufficiently large $j$, by Lemma \ref{lemma: uniform convergence of broken bichar}, $\tau(\Gamma^b_j(N))> b$, $\tau(\Gamma^b_j(-N)) < a$, and $\Gamma^b_j([-N, N])$ uniformly close to $\Gamma^{gl}([-N, N])$. Hence $(y'_j, \eta'_j)$ is away from $(y', \eta')$, contradicting $(y', \eta')$ limiting point.
\end{proof}

In the following, we will denote broken bicharacteristics as $\Gamma^b$, broken null-geodesics as $\gamma^b$, gliding rays as $\Gamma^{gl}$, and the projection of gliding rays as $\gamma^{gl}$.

\section{DN Map as an FIO}\label{sec: DN map}

In this section, we study the microlocal property of the DN map $\Lambda^{U,V}_{g,A,q}$, assuming $U$ and $V$ are disjoint.
We will first recall the optics construction in \cite{Stefanov2018}, which shows that the DN map is an elliptic Fourier Integral Operator. Then we will use microlocal analysis to perform a symbolic construction of the solution as in \cite{Melrose1979}, and explicitly compute the principal symbol.

\subsection{Notation}\label{subsec: notation for dn map}
We first setup some common notations that will be used in both constructions. Let $\Lambda_0' \subset \mathcal{H}$ be a conic Lagrangian submanifold of $T^*\partial M$. We denote by
\[
\Lambda = \bigcup_{(x',\xi') \in \Lambda_0'}\Gamma^b_{x',\xi'}
\]
the union of forward broken null-bicharacteristics $\Gamma^b_{x',\xi'}$ starting from $(x',\xi') \in \Lambda_0'$.
Recall that we define the broken null-bicharacteristic to be a path in $T^*M^{\circ} \sqcup T^*\partial M$ and $\tau$ is the temporal function.
By assuming the base projection of $\Lambda_0'$ is sufficiently small and focusing only on a compact region $\tau^{-1}([-T, T])$, we may decompose $\Lambda$ into
\[
\Lambda = \left(\bigsqcup_{j=0}^N \Lambda_j\right) \sqcup \left(\bigsqcup_{j=0}^N \Lambda_j'\right)
\]
where each $\Lambda_j \subset T^*M^{\circ}$ is union of null-bicharacteristics from boundary to boundary, and each $\Lambda_j' \subset T^*\partial M$ is the reflection part that connects $\Lambda_j$ with $\Lambda_{j+1}$. Specifically,
\[
\pi(\Lambda_j|_{\partial M}) = \Lambda_j' \sqcup \Lambda_{j+1}', \quad L(\Lambda_j') = \Lambda_{j+1}',
\]
where $\pi: T^*M|_{\partial M} \to T^*\partial M$ and $L$ is the lens relation.
The set $\Lambda_0$, as the flowout of forward bicharacteristics from $\Lambda_0'$, is a Lagrangian submanifold of $T^*M^{\circ}$.
The reflection part $\Lambda_j' = L^j(\Lambda'_0)$ is also a Lagrangian submanifold, since $L$ on an admissible Lorentzian manifold locally gives canonical graph\footnote{The graph of $L$ is locally a canonical graph when the boundary is strictly null-convex since all bicharacteristics hits boundary transversally, see \cite[Proposition 2.4]{Hintz2017} and \cite[Theorem 4.2]{Stefanov2018}.}.
For simplicity, let $U_j \subset \partial M$ be an open neighborhood of the base projection of $\Lambda_j'$ such that $\pi(\Lambda_j|_{U_j}) = \Lambda_j'$ and $\pi(\Lambda_j|_{U_{j+1}}) = \Lambda_{j+1}'$.
We may assume  $\Lambda_0'$ is sufficiently small such that $U_j$ are then disjoint from each other.

By \cite[Lemma 2.2]{Hintz2017}, one can extend $(M,g)$ to a larger Lorentzian manifold $(M_e, \tilde{g})$, such that $M$ is contained in the interior of $M_e$ and $\tilde{g}|_M = g$.
Recall that $\Lambda_j$ is the union of bicharacteristics intersecting the boundary transversally.
Then we may extend $\Lambda_j$ across the boundary into $T^*M_e$.
By a slight abuse of notation, we continue to write this extension as $\Lambda_j$, and we require that $\Lambda_j \cap \Lambda_k = \emptyset$.
By slightly extending $M$, we may assume that bicharacteristics in $\Lambda_j$ do not return to $M$, to avoid complications.
Likewise, we can extend $A$ and $q$ to $\tilde{A}$ and $\tilde{q}$ on $M_e$.
We emphasize that this extension is purely for the convenience of constructing solutions, and that the result is independent of the specific extension chosen.

In addition, we denote by $\Gamma^b$ a forward broken bicharacteristic in $\Lambda$, with boundary points
\begin{equation}\label{def: Gammab}
(x_0', \xi^{0\prime}) \to (x_1', \xi^{1\prime}) \to \cdots \to (x_N', \xi^{N\prime}), \quad (x_j', \xi^{j\prime}) \in \mathcal{H} \cap T^*U_j.
\end{equation}
We will perform optics construction near $\Gamma^b$, and symbolic construction on $\Lambda$, with $\Lambda$ contained in a neighborhood of $\Gamma^b$.

\subsection{Optics construction}\label{subsec: optics}

In this subsection, we briefly recall the optics construction near the chosen broken bicharacteristic $\Gamma^b$, and recall Theorem 4.2 of \cite{Stefanov2018}, for more details see \cite{Stefanov2018}. Let $f \in \mathcal{E}'(U_0)$ be such that $WF(f) \subset \mathcal{H}$ is a conic neighborhood of $(x_0', \xi^{0\prime})$. The solution near $x_0'$ has the form
\[
u(x) = (2\pi)^{-n}\int e^{i\phi(x,\theta)}a(x,\theta)\hat{f}(\theta)d\theta.
\]
Here in semi-geodesic normal coordinates, the phase function $\phi$ solves the Eikonal equation
\begin{equation}\label{eq: eikonal}
g^{\alpha\beta}\partial_{\alpha}\phi\partial_{\beta}\phi + (\partial_n\phi)^2 = 0.
\end{equation}
The amplitude is of the form $a \sim \sum_{j=0}^{\infty}a_j(x, \theta)$ for $a_j$ homogeneous of degree $-j$ in $\theta$, and they inductively solve
\begin{align*}
    Ta_0 = 0, \quad &a_0|_{x_n = 0} = \chi,\\
    iTa_j = -\Box_{g,A,q}a_{j-1}, \quad &a_j|_{x_n=0} = 0, \quad j\geq 1
\end{align*}
where $\chi(x,\theta)$ is a conic cutoff around $(x_0',\xi^{0\prime})$ that is homogeneous of degree 0 in $\theta$, and $T$ is a transport operator along bicharacteristics
\[
T = 2g^{jk}\partial_j\phi(\partial_k - iA_k) + \Box_g\phi.
\]
To extend out the solution $u$, one can take a timelike surface transversal to $\Gamma^b$, take the restriction of $u$ on that surface, and construct the optics solution again. Repeated this construction, one can thus obtain the optics solution along the entire bicharacteristic until it hits $x_1'$.

Around $x_1'$, the solution is the sum of incident wave and reflection wave $u = u_1^\mathrm{inc} + u_1^\mathrm{ref}$. The two solutions satisfy
\begin{align*}
    &u^\mathrm{inc}(x) = (2\pi)^{-n}\int e^{i\phi(x, \theta)}(a_0^\mathrm{inc}+a_1^\mathrm{inc}+R^\mathrm{inc})(x, \theta)\hat{f}(\theta)d\theta,\\
    &u^\mathrm{ref}(x) = (2\pi)^{-n}\int e^{i\phi^\mathrm{ref}(x, \theta)}(a_0^\mathrm{ref}+a_1^\mathrm{ref}+R^\mathrm{ref})(x, \theta)\hat{f}(\theta)d\theta.
\end{align*}
Both phase functions satisfy the Eikonal equation, and both symbols satisfy the transport equation, as stated above. They are related by
\begin{align*}
    &\phi|_{U_1} = \phi^\mathrm{ref}|_{U_1}, \quad \partial_{\nu}\phi|_{U_1} = -\partial_{\nu}\phi^\mathrm{ref}|_{U_1}\\
    &a_0^\mathrm{ref}|_{U_1} = -a_0^\mathrm{inc}|_{U_1}, \quad a_1^\mathrm{ref}|_{U_1} = -a_1^\mathrm{inc}|_{U_1}.
\end{align*}
The DN map $\Lambda^{U_0, U_1}_{g,A,q}$ gives
\[
\Lambda^{U_0, U_1}_{g,A,q}f = (2\pi)^{-n}\int e^{i\phi}(2i(\partial_{\nu}\phi)a_0^\mathrm{inc}+ 2i(\partial_{\nu}\phi)a_1^\mathrm{inc}+ \partial_{\nu}(a_0^\mathrm{inc}+a_0^\mathrm{ref})+ a_{-1})\hat{f}(\theta)d\theta.
\]
Hence, {it is proved in \cite[Theorem 4.2]{Stefanov2018} that} the restricted DN map $\Lambda^{U_0, U_1}_{g,A,q}$ is an elliptic Fourier Integral Operator of order $1$ associated with the canonical graph
\[
\{(L(x',\xi'); x', \xi'): (x', \xi') \in \mathcal{H} \cap T^*U_0\} \subset T^*U_1 \times T^*U_0.
\]
Furthermore, one can inductively construct solutions $\uinc_j$ and $\uref_j$ along $\Gamma^b$, for $j = 2, \ldots, N$.
Note that by the proof of \cite[Theorem 4.2]{Stefanov2018}, the map $f \mapsto u_1^\mathrm{inc}|_{U_1}$ is an elliptic Fourier Integral Operator of order $0$ and so is the map $F_j: -u_{j-1}^\mathrm{ref}|_{U_{j-1}} \mapsto u_{j}^\mathrm{inc}|_{U_{j}}$, for $j = 2, \ldots, N$.
A similar argument together with the transversal intersection calculus of Fourier Integral Operators indicates $\Lambda^{U_0, U_k}_{g,A,q}$ is an elliptic Fourier Integral Operator associated with the canonical graph
\[
\{(L^k(x',\xi'); x', \xi'): (x', \xi') \in \mathcal{H} \cap T^*U_0\} \subset T^*U_k \times T^*U_0.
\]
We emphasize again that the above result is essentially proved in \cite[Theorem 4.2]{Stefanov2018}. Since it is elliptic, we can obtain the following propagation of singularity result for the global DN map $\Lambda_{g,A,q}$ (that is when $U = V = \pM$).
\begin{lem}\label{lemma: propagation of singularity}
    Suppose $WF(f) = \{(x', t\xi'): t> 0\}$, and view $WF(\Lambda_{g,A,q}f)$ as subset of $S^*\partial M = T^*\partial M \backslash 0 / \mathbb{R}^+$.
    We have the following statements.
    \begin{itemize}
        \item If $WF(f) \subset \mathcal{E}$,  then $WF(\Lambda_{g,A,q}f) \backslash WF(f)$ is empty;
        \item If $WF(f) \subset \mathcal{G}$, then $WF(\Lambda_{g,A,q}f) \backslash WF(f)$ is empty or a continuous set of points;
        \item If $WF(f) \subset \mathcal{H}$, then $WF(\Lambda_{g,A,q}f) \backslash WF(f)$ is a discrete set of points given by the boundary points of the corresponding forward broken bicharacteristic.
    \end{itemize}
\end{lem}
\begin{proof}
    For the elliptic region, by \cite[(24.2.4)]{MR2304165} and \cite[Theorem 18.3.27]{MR2304165}, we have
    \[
    \mathcal{E} \cap WF(\Lambda_{g,A,q}f) \subset \mathcal{E} \cap WF(f).
    \]

    For the glancing region, we use Lemma \ref{lemma: glancing is gliding} and the propagation of singularity result \cite[Theorem 0.4]{Melrose1977}. The above two results show that $WF(\Lambda_{g,A,q}f)\backslash WF(f)$ contains a point of the glancing set if and only if it contains the entire gliding ray. In particular, if $WF(f)$ is disjoint from $\mathcal{G}$, then by Lemma \ref{lemma: gliding rays are tame} and the fact that $\supp(u) \subset J^+(\supp(f))$, $WF(\Lambda_{g,A,q}f)$ is disjoint from $\mathcal{G}$.

    For the hyperbolic region, suppose $WF(f) \subset \mathcal{H}$. By the optics construction and the fact that each $\Lambda^{U_0,U_k}_{g,A,q}$ is an elliptic FIO with canonical relation given by powers of lens relation, $WF(\Lambda_{g,A,q}f)\backslash WF(f)$ is the boundary points of the corresponding forward broken bicharacteristics from $WF(f)$. If $WF(f)$ is disjoint from $\mathcal{H}$, Proposition \ref{prop: broken null-geod are tame} and \cite[Theorem 24.2.1]{MR2304165} show that $WF(\Lambda_{g,A,q}f)$ is disjoint from $\mathcal{H}$.


\end{proof}

Throughout the paper, we shall view the wavefront set as both a conic subset of the cotangent bundle ($T^*\partial M \backslash 0$) and a subset of the spherical bundle ($S^*\partial M$). When referring to a point in the wavefront set, we always view the wavefront set as a subset of the spherical bundle.

\subsection{Microlocal construction of solution}\label{subsec: microlocal construction}

In this subsection, we construct solutions microlocally so that it is easier to analyze the evolution of principal symbol globally.

Let $f \in I^{\mu}_{cl}(\Lambda_0')$, denote $\tilde{f} = f|\bar{g}|^{1/4} \in I^{\mu}_{cl}(\Lambda_0'; \Omega^{1/2}_{\partial M})$ where $\bar{g} = g|_{T\partial M \times T\partial M}$. We shall perform a symbolic construction of the solution $\tilde{u}_0 = u_0|\tilde{g}|^{1/4} \in I^{\mu+1/4}_{cl}(\Lambda_0; \Omega^{1/2}_{M_e})$ as in \cite{Melrose1979}. Specifically, we inductively construct
\[
\tilde{v}_k = v_k|\tilde{g}|^{1/4} \in I^{\mu+1/4-k}_{cl}(\Lambda_0; \Omega^{1/2}_{M_e})
\]
such that the following equation is solved on the principal level
\begin{equation}\label{eq: symbolic construction}
    \begin{aligned}
        P_{\tilde{g},\tilde{A},\tilde{q}}(\tilde{v}_0+ \cdots + \tilde{v}_k) &\in I^{\mu-3/4-k}_{cl}(\Lambda_0; \Omega^{1/2}_{M_e}) &\text{ in } M_e,\\
        f - (v_0 + \cdots + v_k)|_{U_0} &\in I^{\mu-k-1}_{cl}(\Lambda_0') &\text{ on } U_0.
    \end{aligned}
\end{equation}
With this construction, let $\tilde{u}_0 = u_0 |\tilde{g}|^{1/4}$ be some asymptotic sum of $\tilde{v}_0 + \tilde{v}_1 + \cdots$, it solves
\begin{equation*}
    \begin{aligned}
        \Box_{\tilde{g},\tilde{A},\tilde{q}}u_0 &\in C^{\infty}(M_e),\\
        u_0|_{U_0} - f &\in C^{\infty}(U_0).
    \end{aligned}
\end{equation*}

First, we start with some $\tilde{v}_0 = v_0|\tilde{g}|^{1/4} \in I^{\mu+1/4}_{cl}(\Lambda_0; \Omega^{1/2}_{M_e})$.
By equation \eqref{eq: lie derivative of principal symbol} and \eqref{eq: symbol of P}, we choose ${v}_0$ such that
\[
0 = i\sigma(P_{\tilde{g},\tilde{A},\tilde{q}}\tilde{v}_0) = \left(\mathcal{L}_{H_p}-iA(H_p)\right)\sigma(\tilde{v}_0).
\]
Note that this is a first order differential equation for $\sigma(\tilde{v}_0)$ on $\Lambda_0$. The initial condition is given by $\sigma(\tilde{f})$ via the following lemma.
\begin{lem}\label{lemma: relate principal symbol with restriction}
    Let $\Lambda_0$, $\Lambda_0'$, and $U_0$ be defined above.
    Suppose  $v|_{U_0}= f$, where $v \in I^{\mu+1/4}_{cl}(\Lambda_0)$ and $f\in I^{\mu}_{cl}(\Lambda_0')$.
    Let $(x', x^n)$ be the semi-geodesic normal coordinates and $\pi: \Lambda_0|_{U_0} \to \Lambda_0'$ denote the natural projection.
    For $(x, \xi) \in \Lambda_0|_{U_0}$, let $w_1, \dots, w_n \in T_{(x,\xi)}\Lambda_0$ be such that  $dx^n(w_n) \neq 0$.
    Let $\tilde{v} = v|\tilde{g}|^{1/4}$ and $\tilde{f} = f|\bar{g}|^{1/4}$ be the corresponding half-densities.
    Then there exists $a_j \in \mathbb{R}$ such that
    \[
    dx^n(w_j + a_jw_n) = 0 \quad \text{ for } j = 1, \ldots, n-1,
    \] and the principal symbols satisfy
    \[
    \sigma(\tilde{v})(w_1, \dots, w_n) = e^{i\pi/4}\sigma(\tilde{f})(\pi_*(w_1+a_1w_n), \dots, \pi_*(w_{n-1}+a_{n-1}w_n))|dx^n(w_n)|^{1/2}.
    \]
\end{lem}
\begin{proof}

    To see how $\sigma(\tilde{v})$ is related to $\sigma(\tilde{f})$, we consider the local representations of $v$ and $f$.
    Indeed, in the semi-geodesic normal coordinates, locally $f$ can be written as an oscillatory integral of the form
    \[
     f(x') = \int e^{i\phi_0(x', \theta)}a_0(x', \theta)d\theta + l.o.t.,
     \]
    where $\phi_0$ is non-degenerate, $a_0$ is homogeneous in $\theta$, and $l.o.t.$ denotes lower order terms.
    Here by non-degenerate, we mean the differentials $d_{x'}(\partial_{\theta_j}\phi_0)$, for $j = 1, \ldots, N$, are linearly independent, see \cite[Section 1.2]{Hoermander1971}.
    Note that $\phi_0$ parameterizes $\Lambda_0'$ in the sense that $\xi' = \partial_{x'} \phi_0(x', \theta)$, for $(x', \xi') \in \Lambda_0'$ and $\partial_{\theta} \phi_0(x', \theta) =0$.
    With $\Lambda_0' \subset \mH$, the covector $\xi'$ is timelike and therefore we have
    $
    g^{\alpha\beta} \partial_{x^\alpha} \phi_0 \partial_{x^\beta} \phi_0 < 0,
    $
    for $1 \leq \alpha, \beta \leq n-1$.
    Thus, there exists a smooth function $\phi(x, \theta)$ locally solving the Eikonal equation in (\ref{eq: eikonal}), such that
    \[
    \phi(x', 0, \theta) = \phi_0(x', \theta) \quad \text{ and } \quad \partial_{x^n} \phi(x', 0, \theta) \neq 0.
    \]
    Such $\phi$ parameterizes $\Lambda_0$ locally. With $f(x') = v(x',0)$,
    we may write
    \[
     v(x) = \int e^{i\phi(x, \theta)}a(x, \theta)d\theta + l.o.t.,
     \]
    where $a(x', 0, \theta) = a_0(x', \theta)$ is the leading amplitude.
    We emphasize, near $x^n = 0$, the phase function $\phi$ is also non-degenerate, as $d_{x}(\partial_{\theta_j}\phi_0)$, for $j = 1, \ldots, N$, are linearly independent at $x^n=0$.

    Let $\lambda'$ be local coordinates on $C_0 = \{(x', \theta): {\partial_\theta \phi}(x',0,\theta)=0\}$, then by stationary phase computation the coordinate dependent principal symbol of $f$ is given by
    \[
    \sigma_{x'}(f) = (F_0^{-1})^*\left( e^{i\pi N/4}a(x', 0, \theta) \left|\frac{D(\lambda', {\partial_\theta \phi}(\cdot,0,\cdot))}{D(x', \theta)}\right|^{-1/2}|d\lambda'|^{1/2} \right)
    \]
    where $F_0: C_0 \to \Lambda_0', (x', \theta) \mapsto (x', 0, \phi_{x'}(x',0,\theta),0)$ is locally a diffeomorphism, and $N$ is the signature. Note that
    \[
    C = \{(x, \theta): {\partial_\theta \phi} = 0\} = C_0 \sqcup \{(x, \theta): {\partial_\theta \phi} = 0, x^n \neq 0\},
    \]
    so one can extend $\lambda'$ to $C$ such that $(\lambda', x^n)$ form local coordinates around $C_0$. This gives, on $\partial M$,
    \[
    \sigma_x(v) = (F^{-1})^*\left( e^{i\pi (N+1)/4}a(x', 0, \theta) \left|\frac{D(\lambda', {\partial_\theta \phi}(\cdot,0,\cdot))}{D(x', \theta)}\right|^{-1/2}|d\lambda'\wedge dx^n|^{1/2} \right)
    \]
    where $F: C \to \Lambda_0, (x, \theta) \mapsto (x, \phi_x)$ and we used the fact that on $\partial M$
    \[
    \frac{D(\lambda', x^n, {\partial_\theta \phi})}{D(x', x^n, \theta)} = \begin{bmatrix}
        \frac{D(\lambda', {\partial_\theta \phi}(\cdot,0,\cdot))}{D(x', \theta)} & \left.\frac{D(\lambda', {\partial_\theta \phi})}{Dx^n}\right|_{x^n=0}\\
        0 & 1
    \end{bmatrix}.
    \]

    First suppose $w_1,\dots, w_{n-1}$ are tangent to $\Lambda_0|_{U_0}$, which is equivalent to $dx^n(w_j) = 0$. Then $dx^n(F_*^{-1}w_j) = (F^{-1})^*dx^n (w_j) = dx^n w_j = 0$. We also have $F^{-1}_*w_j = (F^{-1}_0)_*(\pi_*w_j)$ because $F^{-1}(\Lambda_0|_{U_0}) = C_0 = F^{-1}_0(\Lambda_0')$. Then we have
    \begin{align*}
        &\sigma_x(v)(w_1, \dots, w_n) \\
        &= e^{i\pi(N+1)/4}a(x', 0, \theta) \left|\frac{D(\lambda', {\partial_\theta \phi}(\cdot,0,\cdot))}{D(x', \theta)}\right|^{-1/2}|d\lambda'\wedge dx^n|^{1/2} (F^{-1}_*w_1, \dots, F^{-1}_*w_n)\\
        &= e^{i\pi (N+1)/4}a(x', 0, \theta) \left|\frac{D(\lambda', {\partial_\theta \phi}(\cdot,0,\cdot))}{D(x', \theta)}\right|^{-1/2}\\
        &\quad\quad\cdot|d\lambda'|^{1/2} ((F^{-1}_0)_*\pi_*w_1, \dots, (F^{-1}_0)_*\pi_*w_{n-1})|dx^n(w_n)|^{1/2}\\
        &= e^{i\pi/4}\sigma_{x'}(f)(\pi_*w_1, \dots, \pi_*w_{n-1})|dx^n(w_n)|^{1/2}.
    \end{align*}
    In semi-geodesic normal coordinates, $|g|_x = |\bar{g}|_{x'}$, so
    \[
    \sigma(\tilde{v})(w_1, \dots, w_n) = e^{i\pi/4}\sigma(\tilde{f})(\pi_*w_1, \dots, \pi_*w_{n-1})|dx^n(w_n)|^{1/2}.
    \]

    Now consider more general $w_1, \dots, w_{n-1}$, let $w_n$ be such that $dx^n(w_n) \neq 0$. By setting $a_j = -\frac{dx^n(w_j)}{dx^n(w_n)}$ we have $dx^n(w_j +a_jw_n) = 0$. We can thus apply the previous result and get
    \begin{align*}
        \sigma(\tilde{v})(w_1, \dots, w_n) &= \sigma(\tilde{v})(w_1+a_1w_n, \dots, w_{n-1}+a_{n-1}w_n, w_n)\\
        &= e^{i\pi/4}\sigma(\tilde{f})(\pi_*(w_1+a_1w_n), \dots, \pi_*(w_{n-1}+a_{n-1}w_n))|dx^n(w_n)|^{1/2}
    \end{align*}
    where the first equal sign comes from
    \[
    w_1 \wedge \dots \wedge w_n = (w_1+a_1w_n) \wedge \dots \wedge (w_{n-1}+a_{n-1}w_n) \wedge w_n.
    \]
\end{proof}

We can set $w_n = H_p$, as it is always transversal to the boundary on broken bicharacteristics, that is $|dx^n(H_p)| = 2|\xi_n| > 0$ (recall that $\xi_n = \sqrt{-g^{\alpha\beta}\xi_{\alpha}\xi_{\beta}}$ is already non-negative, we use $|\xi_n|$ to emphasize this). By Lemma \ref{lemma: relate principal symbol with restriction}, $\sigma(\tilde{v}_0)$ solves
\[
\begin{cases}
    \left(\mathcal{L}_{H_p}-iA(H_p)\right)\sigma(\tilde{v}_0) = 0 \quad \text{on } \Lambda_0,\\
    \sigma(\tilde{v}_0)(w_1, \dots, w_n) = e^{i\pi/4}\sigma(\tilde{f})(\pi_*(w_1+a_1w_n), \dots, \pi_*(w_{n-1}+a_{n-1}w_n))|dx^n(w_n)|^{1/2},
\end{cases}
\]
where $(x',x^n)$ is semi-geodesic normal coordinates, $w_j$ are vectors at $\Lambda_0|_{U_0}$ such that $dx^n(w_n) \neq 0$, and $a_j = -\frac{dx^n(w_j)}{dx^n(H_p)}$. Since $f$ is set to be classical,
{
this transport equation has a unique homogeneous solution, for example see \cite[Section 6.4]{Duistermaat1972} and \cite[Section 6]{Melrose1979}.
}

Choose any $\tilde{v}_0$ such that $\sigma(\tilde{v}_0)$ solves the above equation, and denote $h_1 = P_{\tilde{g},\tilde{A},\tilde{q}}\tilde{v}_0 \in I^{\mu-3/4}_{cl}(\Lambda_0; \Omega^{1/2}_{M_e})$. By construction, $v_0|_{U_0}$ is a Lagrangian distribution with respect to $\Lambda_0'$, and $f - v_0|_{U_0} \in I_{cl}^{\mu-1}(\Lambda_0')$ in $U_0$. Consider $\tilde{v}_1 = v_1|\tilde{g}|^{1/4} \in I^{\mu-3/4}(\Lambda_0; \Omega^{1/2}_{M_e})$, \eqref{eq: symbolic construction} implies that it solves
\[
\begin{cases}
    P_{\tilde{g}, \tilde{A},\tilde{q}}\tilde{v}_1 = -h_1 + I_{cl}^{\mu-7/4}(\Lambda_0; \Omega^{1/2}_{M_e}),\\
    v_1|_{U_0} = f - v_0|_{U_0} + I_{cl}^{\mu-2}(\Lambda_0').
\end{cases}
\]
We again use \eqref{eq: symbol of P}, \eqref{eq: lie derivative of principal symbol} and Lemma \ref{lemma: relate principal symbol with restriction} to solve the equation for $\sigma(\tilde{v}_1)$. The lower order terms $\tilde{v}_k$ can be constructed inductively, and the construction is complete.

Note that by construction, $u_0|_{U_1} \in I^{\mu}_{cl}(\Lambda_1')$ is a Lagrangian distribution by Proposition \ref{prop: Lambda_restriciton}. Thus, we again inductively and symbolically construct $\tilde{u}_k \in I^{\mu+1/4}_{cl}(\Lambda_k; \Omega^{1/2}_{M_e})$ in the same way, and require it to solve
\[
\begin{cases}
    P_{\tilde{g},\tilde{A},\tilde{q}}\tilde{u}_k \in C^{\infty}(M_e),\\
    u_k|_{U_k} + u_{k-1}|_{U_k} \in C^{\infty}(U_k).
\end{cases}
\]
Put everything together and restrict to $\tau^{-1}([-T, T]) \subset M$, $\tilde{u} = \sum_{j=0}^N \tilde{u}_j$ thus solves
\begin{equation*}
    \begin{cases}
        \Box_{g, A, q}u \in C^{\infty}(M^{\circ} \cap \tau^{-1}([-T, T])),\\
        u|_{\partial M} = f + C^{\infty}(\partial M \cap \tau^{-1}([-T, T])).
    \end{cases}
\end{equation*}
In other words, $u$ differs from the actual solution only by a smooth function, so the actual solution is also a Lagrangian distribution sharing the same principal symbol as $u$. Next we analyze what the principal symbol of $u$ is.
\begin{prop}\label{prop: relate principal symbol along bich flow}
    Let $\Lambda_{beg}$ be a small conic Lagrangian submanifold in $\mathcal{H}$, denote
    \[
    \Lambda_{flow} = \cup_{(y', \eta') \in \Lambda_{beg}} \Gamma_{y',\eta'}
    \]
    its flowout Lagrangian along forward bicharacteristics. Let
    \[
    \Lambda_{end} = \{\Gamma_{y', \eta'}(t) \in \partial M: \Gamma_{y',\eta'}(0) = (y', \eta') \in \Lambda_{beg}, t \neq 0\}
    \]
    the projection of ending points of these bicharacteristics. When $\Lambda_{beg}$ is small enough, the projection of $\Lambda_{beg}$ and $\Lambda_{end}$ on $\partial M$ are disjoint. Then we have the lens relation
    \[
    L_{\Lambda_{flow}}: \Lambda_{beg} \ni (y', \eta') \mapsto (x', \xi') \in \Lambda_{end}
    \]
    along bicharacteristics. Suppose $\tilde{v} = v|g|^{1/4} \in I_{cl}(\Lambda_{flow}; \Omega_M^{1/2})$ satisfies
    \[
    \mathcal{L}_{H_p}\sigma(\tilde{v}) = iA(H_p)\sigma(\tilde{v}),
    \]
    denote $\tilde{v}_1 = v|_{U_1}|\bar{g}|^{1/4} \in I_{cl}(\Lambda_{beg}; \Omega_{\partial M}^{1/2})$, $\tilde{v}_2 = v|_{U_2}|\bar{g}|^{1/4} \in I_{cl}(\Lambda_{end}; \Omega_{\partial M}^{1/2})$ where $U_1, U_2 \subset \partial M$ disjoint small neighborhoods around $\Lambda_{beg}, \Lambda_{end}$, respectively. Then we have
    \[
    |\xi_n|^{1/2}\sigma(\tilde{v}_2)|_{(x', \xi')} = \exp{\left(i\int_{\gamma}A(\dot{\gamma})\right)} \cdot |\eta_n|^{1/2}(L_{\Lambda_{flow}}^{-1})^*(\sigma(\tilde{v}_1)|_{(y', \eta')}),
    \]
    where the null-geodesic $\gamma$ is the projection of forward bicharacteristic from $(y', \eta')$ to $(x', \xi')$ (note that $\int_\gamma A(\dot{\gamma})$ is independent of parametrization of $\gamma$), and
    \[
    \xi_n = \sqrt{-\bar{g}^{\alpha\beta}\xi_{\alpha}\xi_{\beta}}, \quad \eta_n = \sqrt{-\bar{g}^{\alpha\beta}\eta_{\alpha}\eta_{\beta}}.
    \]
\end{prop}

To prove the proposition, we first state a Lemma, relating the pushforward of vectors along the flow and the pushforward of vectors via lens relation.

\begin{lem}\label{lemma: pushforward of vector along flow}
    Let $L = L_{\Lambda_{flow}}, \Lambda_{beg}, \Lambda_{flow}, \Lambda_{end}$ be as defined in Proposition \ref{prop: relate principal symbol along bich flow}. Denote $\Phi_t$ the flow of $H_p$ along $\Lambda_{flow}$. Let $L(\bar{y}', \bar{\eta}') = (\bar{x}', \bar{\xi}')$, $w \in T_{(\bar{y}', \bar{\eta}')}\Lambda_{beg}$, $\pi: \Lambda_{flow}|_{\partial M} \to \Lambda_{beg}\sqcup \Lambda_{end}$, then
    \[
    L_*w = \pi_*((\Phi_T)_*(\pi^{-1}_*w) + aH_p)
    \]
    where $\Phi_T(\pi^{-1}(\bar{y}',\bar{\eta}')) = \pi^{-1}(\bar{x}', \bar{\xi}')$ and $a \in \mathbb{R}$ is the unique number such that
    \[
    dx^n(\Phi_{T*}(\pi^{-1}_*w) + aH_p) = 0.
    \]
\end{lem}
\begin{proof}
    Let $t(y', \eta')$ be the arriving time for $(y', \eta') \in \Lambda_{beg}$, that is
    \[
    \Phi_{t(y', \eta')}(y', 0, \eta_\pm) = (x', 0, \xi_\mp) \in \pi^{-1}\Lambda_{end}, \quad L = \pi \circ \Phi_{t(\cdot)} \circ \pi^{-1}.
    \]
    Note that $t(\bar{y}', \bar{\eta}') = T$. Denote $\tilde{w} = (\Phi_T)_*\pi^{-1}_*w$, it suffices to show
    \[
    \tilde{w} - (\Phi_{t(\cdot)-T})_*\tilde{w} = aH_p
    \]
    for some $a \in \mathbb{R}$. Denote $N = \{\Phi_T(y', 0, \eta_\pm): (y', \eta') \in \Lambda_{beg}\} = \Phi_T(\pi^{-1}\Lambda_{beg})$, $\tilde{N} = \{\Phi_{t(y', \eta')}(y', 0, \eta_\pm): (y', \eta') \in \Lambda_{beg}\} = \pi^{-1}\Lambda_{end}$, clearly $\pi^{-1}(\bar{x}', \bar{\xi}') \in N\cap \tilde{N}$. $\tilde{N}$ can be viewed as a graph of $N$ with coordinates $(p, t)$ where $p \in N$ and $\Phi_t(p) \in \tilde{N}$. The result then immediately follows by computing in local coordinates, since in local coordinates the pushforward of $\partial_{z^1}$ from $\{(z, t): t=0\}$ to $\{(z, t): t = f(z)\}$ via $(\cdot ,f(\cdot))$ is simply $\partial_{z^1} + \frac{\partial f}{\partial z^1}\partial_t$, and here $\partial_t$ is given by $H_p$. Finally, $a$ is uniquely determined by the fact that $\tilde{w} \in TN$ and $(\Phi_{t(\cdot)-T})_*\tilde{w} \in T\tilde{N}$.
\end{proof}

\begin{proof}[Proof of Proposition \ref{prop: relate principal symbol along bich flow}]
    Solve the Lie derivative equation one get
    \[
    \Phi_T^*(\sigma(\tilde{v})|_{(x',0,\xi_\pm)}) = \exp\left(i\int_{\gamma}A(\dot{\gamma})\right)\sigma(\tilde{v})|_{(y',0, \eta_\mp)}
    \]
    Denote $w_1, \cdots, w_{n-1} \in T_{(y', \eta')}\Lambda_{beg}$. Then
    \begin{align*}
        &|\xi_n|^{1/2}\sigma(\tilde{v}_2)|_{(x', \xi')}(L_*w_1, \cdots, L_*w_{n-1})\\
        &=|\xi_n|^{1/2}\sigma(\tilde{v}_2)|_{(x', \xi')}(\pi_*((\Phi_T)_*(\pi^{-1}_*w_1) + a_1H_p), \cdots, \pi_*((\Phi_T)_*(\pi^{-1}_*w_{n-1}) + a_{n-1}H_p))\\
        &=C\sigma(\tilde{v})|_{(x',0,\xi_\pm)}(H_p, (\Phi_T)_*(\pi^{-1}_*w_1) + a_1H_p, \cdots, (\Phi_T)_*(\pi^{-1}_*w_{n-1}) + a_{n-1}H_p)\\
        &=C\sigma(\tilde{v})|_{(x',0,\xi_\pm}((\Phi_T)_*H_p, (\Phi_T)_*(\pi^{-1}_*w_1), \cdots, (\Phi_T)_*(\pi^{-1}_*w_{n-1}))\\
        &=C\Phi_T^*(\sigma(\tilde{v})|_{(x',0,\xi_\pm)})(H_p, \pi^{-1}_*w_1, \cdots, \pi^{-1}_*w_{n-1})\\
        &=C\exp{(i\int_{\gamma}A(\dot{\gamma}))}\cdot\sigma(\tilde{v})|_{(y',0, \eta_\mp)}(H_p, \pi^{-1}_*w_1, \cdots, \pi^{-1}_*w_{n-1})\\
        &=\exp{(i\int_{\gamma}A(\dot{\gamma}))} \cdot |\eta_n|^{1/2}\sigma(\tilde{v}_1)|_{(y', \eta')}(w_1, \cdots, w_{n-1})
    \end{align*}
    where $C = \frac{e^{-i\pi/4}}{\sqrt{2}}$, and we used Lemma \ref{lemma: relate principal symbol with restriction} and Lemma \ref{lemma: pushforward of vector along flow}.
\end{proof}

Let $(x_j', \xi^{j\prime}) \in \Lambda_j'$ be such that they are connected by a forward broken bicharacteristic $\Gamma^b$ from $(x_0, \xi^{0\prime})$. By Proposition \ref{prop: relate principal symbol along bich flow} we thus have for any $k$,
\begin{align*}
    &|\xi^k_n|^{1/2} \sigma(u_{k-1}|_{U_k}|\bar{g}|^{1/4})|_{(x_k',\xi^{k\prime})}\\
    &= \exp{\left(i\int_{\gamma^b(x_{k-1}' \to x_k')}A(\dot{\gamma}^b)\right)} \cdot |\xi^{k-1}_n|^{1/2}(L^{-1})^* \left(\sigma(u_{k-1}|_{U_{k-1}}|\bar{g}|^{1/4})|_{(x_{k-1}', \xi^{k-1\prime})}\right)
\end{align*}
where $\xi^j_n = \sqrt{-\bar{g}^{\alpha \beta}\xi^j_{\alpha}\xi^j_{\beta}}$, $\gamma^b$ is the projection of $\Gamma^b$, and $L$ is the lens relation which maps $\Lambda_j'$ to $\Lambda_{j+1}'$ for all $j = 0, 1, \dots, N-1$. By construction, $u_k|_{U_k} + u_{k-1}|_{U_k} \in C^{\infty}(U_k)$ and $u_0|_{U_0} = f + C^{\infty}(U_0)$, so we have
\begin{equation*}
    \begin{split}
        &|\xi^k_n|^{1/2} \sigma(u_{k-1}|_{U_k}|\bar{g}|^{1/4})|_{(x_k',\xi^{k\prime})} \\
        &= (-1)^{k-1}\exp{\left(i\int_{\gamma^b(x_0' \to x_k')}A(\dot{\gamma}^b)\right)} \cdot |\xi^0_n|^{1/2}(L^{-k})^*\left(\sigma(\tilde{f})|_{(x_0', \xi^{0\prime})}\right).
    \end{split}
\end{equation*}
Note that in semi-geodesic normal coordinates, $\sigma(\partial_{x^n} \psi)|_{(x, \xi)} = i\xi_n\sigma(\psi)$. Then we have
\begin{align*}
    &\sigma(\partial_{\nu}u|_{U_k}|\bar{g}|^{1/4})|_{(x_k', \xi^{k\prime})} \\
    &= \sigma(-\partial_{x^n}u_{k-1}|_{U_k}|\bar{g}|^{1/4})|_{(x_k', \xi^{k\prime})} + \sigma(-\partial_{x^n}u_k|_{U_k}|\bar{g}|^{1/4})|_{(x_k', \xi^{k\prime})}\\
    &= -(\pm i|\xi^k_n|)\sigma(u_{k-1}|_{U_k}|\bar{g}|^{1/4})|_{(x_k', \xi^{k\prime})} - (\mp i|\xi^k_n|)\sigma(u_k|_{U_k}|\bar{g}|^{1/4})|_{(x_k', \xi^{k\prime})}\\
    &=\mp 2i|\xi^k_n|\sigma(u_{k-1}|_{U_k}|\bar{g}|^{1/4})|_{(x_k', \xi^{k\prime})}\\
    &=\pm 2i(-1)^k\exp{\left(i\int_{\gamma^b(x_0' \to x_k')}A(\dot{\gamma}^b)\right)} \cdot |\xi^k_n|^{1/2}|\xi^0_n|^{1/2}(L^{-k})^*\left(\sigma(\tilde{f})|_{(x_0', \xi^{0\prime})}\right),
\end{align*}
where the sign is $+$/$-$ if $(x_0', \xi^{0\prime})$ is past/future-pointing. We thus proved the following result.
\begin{prop}\label{prop: principal symbol formula}
    View $\Lambda^{U_0,U_k}_{g,A,q}$ as an operator on half-density valued distributions via
    \[
    \Lambda^{U_0,U_k}_{g,A,q}(f|\bar{g}|^{1/4}) := (\Lambda^{U_0,U_k}_{g,A,q}f)|\bar{g}|^{1/4}.
    \]
    Suppose $L^k(x_0', \xi^{0\prime}) = (x_k', \xi^{k\prime})$ for some $k \geq 1$, $x_0' \in U_0$, $x_k' \in U_k$, and $L$ is the lens relation. Then microlocally in a conic neighborhood of $(x_k', \xi^{k\prime}; x_0', \xi^{0\prime}) \in T^*U_k \times T^*U_0$, the operator $\Lambda^{U_0, U_k}_{g,A,q}$ is an elliptic Fourier integral operator of order 1 whose associated canonical relation is given by
    \[
    \{(L^k(x', \xi'); x', \xi'): (x',\xi') \in T^*U_0 \cap \mathcal{H}\}.
    \]
    Its principal symbol is given by
    \[
    \sigma(\Lambda^{U_0,U_k}_{g,A,q})|_{(x_k',\xi^{k\prime}; x_0',\xi^{0\prime})} = \pm 2i(-1)^k\exp\left(i\int_{\gamma^b} A(\dot{\gamma}^b)\right)|\xi^k_n|^{1/2}|\xi^0_n|^{1/2}(L^{-k})^*,
    \]
    where
    \begin{itemize}
        \item $\gamma^b$ is the forward broken null-geodesic that is the projection of the forward broken bicharacteristic connecting
        \[
        (x_0', \xi^{0\prime}) \to (x_1', \xi^{1\prime}) \to \cdots \to (x_k', \xi^{k\prime})
        \]
        (note that $\int_{\gamma^b} A(\dot{\gamma}^b)$ is independent of parametrization of $\gamma^b$);
        \item the sign is $+$/$-$ if $(x_0', \xi^{0\prime})$ is a past/future-pointing covector.
    \end{itemize}
\end{prop}

\begin{rem}
    Note that the way we turn DN map into an operator on half-density valued distributions depends on the metric. Specifically, this means that one should always be careful which of the following situations they are in:
    \begin{enumerate}
        \item If $\Lambda^{U,V}_{g_1,A_1,q_1} = \Lambda^{U,V}_{g_2,A_2,q_2}$ are given as \textbf{operators on distributions}, then to make them equivalent as operators on half-densities, we can let
        \[
        \Lambda^{U,V}_{g_j,A_j,q_j}(\omega) := \Lambda^{U,V}_{g_j,A_j,q_j}(\omega\eta^{-1})\eta
        \]
        for a fixed half-density $\eta$ and $j = 1, 2$.
        \item Similarly, if $\Lambda^{U,V}_{g_1,A_1,q_1} = \Lambda^{U,V}_{g_2,A_2,q_2}$ are given as \textbf{operators on half-density valued distributions}, then to make them equivalent as operators on distributions, we can let
        \[
        \Lambda^{U,V}_{g_j,A_j,q_j}(f) := \Lambda^{U,V}_{g_j,A_j,q_j}(f\eta)\eta^{-1}
        \]
        for a fixed half-density $\eta$ and $j = 1, 2$.
    \end{enumerate}
\end{rem}

Now we explore what information can be obtained from principal symbol of $\Lambda^{U_0, U_k}_{g,A,q}$, recall that in our case, the operators are equivalent when acting on distributions. Suppose we know the conformal class of $\bar{g}$ on $U_0 \cup U_k$, then we can choose some Lorentzian metric $h$ on $U_0 \cup U_k$ in that conformal class, and $\bar{g} = e^{\varphi}h$. Taking absolute value of the principal symbol to get rid of affects from $A$ and time orientation, we obtain a formula to relate the conformal factor at $x_0'$ and $x_k'$:
\begin{equation}\label{eq: conformal factor relation}
    \begin{split}
        &(n-2)\varphi(x_0') - n\varphi(x_k') \\
        &= 4\log\left[ \frac{1}{2}|h(\xi^{0\prime}, \xi^{0\prime})h(\xi^{k\prime},\xi^{k\prime})|^{-1/4} \left|\frac{|h(x_k')|_{x'}^{1/4}\sigma_{x'}(\Lambda^{U_0,U_k}_{g,A,q}f)|_{(x_k',\xi^{k\prime})}}{(L^{-k})^*(|h(x_0')|_{x'}^{1/4}\sigma_{x'}(f)|_{(x_0',\xi^{0\prime})})}\right| \right].
    \end{split}
\end{equation}
It is easy to see that the right hand side is coordinate invariant since both terms in the fraction are principal symbols of half-density valued distributions, and it is scaling invariant with respect to $\xi^{0\prime}$. Moreover, it gives the same value for any $f \in I_{cl}^{\mu}(\Lambda_0')$ as long as $(x_0', \xi^{0\prime}) \in \Lambda_0'$, because we are essentially using the principal symbol of $\Lambda^{U_0,U_k}_{g,A,q}$. In the case where $f$ is not a classical distribution, the right hand side becomes quotient of equivalence classes, the equation then holds as $|\xi^{0\prime}| \to \infty$.

On the other hand, the argument of the principal symbol gives
\begin{equation}\label{eq: 1-form relation}
    \exp\left(i\int_{\gamma^b}A(\dot{\gamma}^b)\right) = \pm i(-1)^k \text{arg}\left( \frac{\sigma_{x'}(\Lambda^{U_0,U_k}_{g,A,q}f)|_{(x_k',\xi^{k\prime})}}{(L^{-k})^*(\sigma_{x'}(f)|_{(x_0',\xi^{0\prime})})} \right)
\end{equation}
where $\text{arg}(z) = \frac{z}{|z|}$, and the sign is $+$/$-$ if $(x_0', \xi^{0\prime})$ is future/past-pointing. We emphasize again that the left hand side is independent of parametrization of $\gamma^b$.

\section{Weak Lens Relation and Conformal Class of Boundary Metric}\label{sec: recovery of weak lens relation}


We first define what weak lens relation is. Intuitively, it is a map between $B_U \subset T^*U \cap \mathcal{H}$ and $B_V \subset T^*V \cap \mathcal{H}$ such that singularities in $B_U$ flow along a forward broken bicharacteristic to points in $B_V$.




\begin{defn}
    The \emph{weak lens relation} is the map
    \[
    \tilde{L}: T^*U \cap \mathcal{H} \supset B_U \mapsto B_V \subset T^*V \cap \mathcal{H},
    \]
    such that
    \begin{enumerate}
        \item $(x', \xi') \in B_U$ and $(y', \eta') \in B_V$ implies $y' \in J^+(x')$;
        \item if $(x', \xi')\in B_U$ and $\Gamma^b_{x',\xi'}$ is its corresponding forward broken bicharacteristic, then $\Gamma^b_{x',\xi'}|_V = B_V$;
        \item $B_U$ is maximal in the sense that there can not be another $\tilde{B}_U \supsetneq B_U$ satisfies the above two conditions.
    \end{enumerate}
\end{defn}

It is not hard to see that the map is well-defined and bijective, as each $B_U$ and $B_V$ correspond to a unique broken bicharacteristic. To see why it is weaker than the standard lens relation $L$, given some $(y', \eta')$ in some $B_U$, there is no guarantee that $L(y', \eta')$ is in $B_U$ or $B_V$. Moreover, $B_U$ and $B_V$ are given as unordered sets, whereas if one knows the lens relation, then points in $B_U$ and $B_V$ can be ordered.
\begin{prop}\label{prop: same weak lens relation}
    $\tilde{L}$ can be constructed from $\Lambda^{U,V}_{g,A,q}$.
\end{prop}
\begin{proof}
    We first consider $\tilde{L}$ as a map on the spherical bundle $S^*\partial M = (T^*\partial M \backslash 0) / \mathbb{R}^+$, that is, we recover the direction relations. As mentioned before, we can view wavefront set as a subset of $S^*\partial M$. For any $(x', \xi') \in S^*\partial M$, set some $f$ such that $WF(f) = \{(x', \xi')\}$. By Lemma \ref{lemma: propagation of singularity}, $WF(\Lambda^{U,V}_{g,A,q}f)$ can only be one of the following three cases:
    \begin{enumerate}
        \item $(x', \xi') \in \mathcal{E}$, $WF(\Lambda^{U,V}_{g,A,q}f)\ = \emptyset$;
        \item $(x', \xi') \in \mathcal{H}$, $WF(\Lambda^{U,V}_{g,A,q}f) = \emptyset$ if the corresponding forward broken bicharacteristic does not pass through $V$, or a discrete set of points if it passes through $V$;
        \item $(x', \xi') \in \mathcal{G}$, $WF(\Lambda^{U,V}_{g,A,q}f) = \emptyset$ or the union of several continuous curves.
    \end{enumerate}
    We collect all $(x', \xi')$ with a discrete set of singularities, these are the hyperbolic directions in $S^*U$ whose corresponding forward broken bicharacteristic passes through $V$. Label the corresponding $f$ as $f_{x', \xi'}$, and define $\tilde{L}$ on $S^*\partial M$ via
    \[
    \tilde{L}^{-1}(WF(\Lambda^{U,V}_{g,A,q}f_{x', \xi'})) = \{(z', \zeta') \in S^*U: WF(\Lambda^{U,V}_{g,A,q}f_{z', \zeta'}) = WF(\Lambda^{U,V}_{g,A,q}f_{x', \xi'})\}.
    \]
    We have thus determined $\tilde{L}$ as a map on $S^*\partial M$.

    Pick some arbitrary Riemannian metric on $T^* \partial M$, so that $S^*\partial M$ has a natural section given by covectors of length 1. For each $(x', \xi')$ in case (2) of length 1, let $(y', \eta') \in WF(\Lambda^{U,V}_{g,A,q}f_{x', \xi'})$ be of length 1. We know $(x', \xi')$ is connected with $(y', \lambda \eta')$ for some $\lambda > 0$ by the forward broken bicharacteristic. To find out what $\lambda$ is, choose a sufficiently small conic neighborhood $N_1 \subset T^*U$ of $(x', \xi')$. Then the union of $WF(\Lambda^{U,V}_{g,A,q}f_{z', \zeta'})$ for $(z', \zeta') \in N_1$ forms disjoint neighborhoods around points in $WF(\Lambda^{U,V}_{g,A,q}f_{x', \xi'})$, denote the one around $(y', \eta')$ as $N_2 \subset T^*V$. When $N_1$ is sufficiently small, and $(y', \lambda \eta')$ is the $k$-th reflection of $(x', \xi')$, then points in $N_2$ are the $k$-th reflection of points in $N_1$. We can thus extend $\lambda$ to be a function on $N_1$ homogeneous of degree 0 in phase variable. By Lemma \ref{lemma: lens relation determined by directions}, $\lambda: N_1 \to \mathbb{R}^+$ can be computed by requiring the graph of lens relation to be a canonical relation (that is, symplectic form vanishes).
\end{proof}

\begin{lem}\label{lemma: lens relation determined by directions}
    If $L:N_1 \to N_2$ is a lens relation for $N_1, N_2 \subset T^*\partial M \cap \mathcal{H}$ conic subsets, and $L': N_1' \to N_2'$ its corresponding map on $S^*\partial M$ via $L'([x, \xi]) = [L(x, \xi)]$, then $L$ and $L'$ uniquely determines each other.
\end{lem}
\begin{proof}
    Suppose $L'$ is given, but there exists $L_1$ and $L_2$ such that both define $L'$. Then there exists $\phi: N_1 \to \mathbb{R}^+$ with $\phi(x, k\xi) = \phi(x, \xi)$ for $k > 0$ such that $L_1(x, \xi) = (y, \eta)$, $L_2(x, \xi) = (y, \phi(x, \xi)\eta)$. Use the fact that lens relation is invertible and homogeneous of degree 1 on phase variable,
    \[
    T : = L_1^{-1} \circ L_2: N_1 \to N_1,  (x, \xi) \mapsto (x, \phi(x, \xi)\xi).
    \]
    $L_1$ and $L_2$ are symplectomorphisms, so $T$ is also a symplectomorphism. In particular, use coordinate $(x, \xi, y, \eta)$ for the graph, the tangent vector of the graph at some $(x, \xi, x, \phi(x, \xi)\xi)$ consists of
    \[
    X_i = \partial_{x^i}+\partial_{y^i}+\sum_j \frac{\partial \phi}{\partial x^i} \partial_{\eta_j}, \quad Y_i = \partial_{\xi_i}+\phi \partial_{\eta_i}+\sum_j \frac{\partial \phi}{\partial \xi_i} \xi_j \partial_{\eta_j}.
    \]
    Then we have
    \[
    (d\xi_k \wedge dx^k - d\eta_k \wedge dy^k)(X_i, Y_i) = -1 + \phi + \frac{\partial \phi}{\partial \xi_i}\xi_i = 0,
    \]
    add up all $i$ gets
    \[
    n(-1+\phi) + \nabla_{\xi}\phi \cdot \xi = 0.
    \]
    But $\phi$ is homogeneous of degree 0 on $\xi$, meaning $\nabla_{\xi} \phi \cdot \xi = 0$, so $\phi \equiv 1$. In other words, $L$ can be explicitly computed from $L'$ by the requirement that symplectic form vanishes. The other direction is trivial.
\end{proof}

\begin{rem}\label{rem: weak lens relation determines k-th lens relation}
    Note that during the proof of Proposition \ref{prop: same weak lens relation}, we do not need to know how small $N_1$ should be, since we are essentially only using the pushforward of $T_{(x',\xi')}(T^*\partial M)$ via the $k$-th lens relation at $(x', \xi')$. In particular, the weak lens relation gives the $k$-th lens relation on sufficiently small $N_1$, this fact will be used later. We shall emphasize that we would not be able to know the exact value of $k$ just from weak lens relation.
\end{rem}

We can then use the weak lens relation to recover the conformal class of the boundary metric on the recoverable set.
\begin{prop}\label{prop: same conformal class on boundary}

    Given $\Lambda^{U,V}_{g,A,q}$, then the recoverable set $\tilde{U} \cup \tilde{V}$ can be determined. Moreover, the conformal class of the boundary metric $\bar{g} = g|_{T\partial M \times T\partial M}$ can be constructed on $\tilde{U} \cup \tilde{V}$.

\end{prop}
\begin{proof}
    By Proposition \ref{prop: same weak lens relation}, $\tilde{L}$ can be constructed, we will use it to construct $\tilde{U}, \tilde{V}$ and the glancing set on $\tilde{U} \cup \tilde{V}$. Note that once the glancing set is constructed, it is the light cone of $\bar{g} = g|_{T\partial M \times T\partial M}$, so the conformal class of $\bar{g}$ is determined on recoverable sets.

    By definition, the domain and range of $\tilde{L}$ are subsets of $\mathcal{H}$, so any limiting point stays in $\mathcal{H} \cup \mathcal{G}$. By abuse of notation, we say $(x', \xi') \in T^*U$ is in domain (range) of $\tilde{L}$, if there exists some $B_U \ni (x', \xi')$ ($B_V \ni (x', \xi')$) in the domain of $\tilde{L}$ (range of $\tilde{L}$), and we denote $\tilde{L}(x', \xi') = \tilde{L}(B_U)$ ($\tilde{L}^{-1}(x', \xi') = \tilde{L}^{-1}(B_V)$). Pick a sequence of $(x'_j, \xi'_j)$ in domain of $\tilde{L}$ that converges to some $(x', \xi')$ not in domain of $\tilde{L}$, and let $B_j = \tilde{L}(x'_j, \xi'_j)$ and $\Gamma^b_j$ the broken bicharacteristic from $(x'_j, \xi'_j)$ (certainly we only consider those limit such that $\xi' \neq 0$). Note that either $(x', \xi') \in \mathcal{H}$, and $\Gamma^b_j$ converges uniformly on bounded time interval to $\Gamma^b$, which is the forward broken bicharacteristic from $(x', \xi')$, by smoothness of Hamiltonian vector field; or $(x', \xi') \in \mathcal{G}$, and $\Gamma^b_j$ converges uniformly on bounded time interval to $\Gamma^{gl}$, the forward gliding ray from $(x', \xi')$, by Lemma \ref{lemma: uniform convergence of broken bichar}. By Proposition \ref{prop: broken null-geod are tame}, broken bicharacteristics are all tame and have finitely many boundary points in any compact set. In the first case, the limiting points of $B_j$ can only be $\Gamma^b|_V$, which must be empty, because otherwise $(x', \xi')$ would be in the domain of $\tilde{L}$; in the second case, the limiting points of $B_j$ can be $\Gamma^{gl}|_V$, which is non-empty and forms a continuous set of points if and only if $\Gamma^{gl}$ passes through $V$. Conversely, if $(x', \xi') \in \mathcal{G}$ and the corresponding gliding ray passes through $V$, then by uniform convergence of broken bicharacteristic in Lemma \ref{lemma: broken bichar converge to gliding on boundary}, any sufficiently close hyperbolic point would be in the domain of $\tilde{L}$. As a result, we can identify all the $(x', \xi') \in \mathcal{G}$ whose forward gliding ray passes through $V$. In particular, this means we can recover $\tilde{U}$. Similarly we can also recover $\tilde{V}$ using $\tilde{L}^{-1}$.

    Moreover, by assumption of recoverable point, if dimension is 2+1 we can recover 2 such directions in the glancing set at each recoverable point, which is the entire glancing set; in higher dimensions, if a gliding ray passes through $V$ then all nearby gliding rays also pass through $V$, so we recover an open set of directions in glancing set for each recoverable point, which then recovers the entire glancing set at the point by analyticity of light cone.
\end{proof}


\section{Upgrade Weak Lens Relation to Lens Relation}\label{sec: recover lens relation}

Weak lens relation can be upgraded to lens relation when the boundary is strictly null-convex and when the conformal class on the boundary is determined. From Proposition \ref{prop: same weak lens relation} and Proposition \ref{prop: same conformal class on boundary}, we may assume $\tilde{L}, \tilde{U}, \tilde{V}$ and the conformal class of $\bar{g}|_{\tilde{U}\cup \tilde{V}}$ are given.
Also as a result of knowing the conformal class, $\mathcal{G}, \mathcal{H}$ and $\mathcal{E}$ are given on $\tilde{U} \cup \tilde{V}$.

\begin{prop}\label{prop: upgrade weak lens relation}
    Let $L$ be the (future-pointing) lens relation for $g$. Suppose $\tilde{L}, \tilde{U}, \tilde{V}$, and the conformal class of $\bar{g} = g|_{T(\tilde{U} \cup \tilde{V}) \times T(\tilde{U} \times \tilde{V})}$ are given (thus $\mathcal{G}|_{\tilde{U}\cup\tilde{V}}$ is also determined).
    Then for any recoverable direction $(x', \xi') \in \mathcal{G}|_{\tilde{U} \cup \tilde{V}}$, there exists a conic neighborhood $W'$ of $(x', \xi')$, $W = W' \cap \mathcal{H}$, such that $L|_W$ can be determined up to time orientation. That is, one can construct some $L'$ on $W$ such that
    \[
    L' = L|_W \quad \text{or} \quad L' = L^{-1}|_W.
    \]

\end{prop}
\begin{proof}
    Assume $x' \in \tilde{U}$, similar argument works for $\tilde{V}$. By assumption, the forward gliding ray passes through $V$. Lemma \ref{lemma: broken bichar converge to gliding on boundary} shows that there exists a conic neighborhood $\tilde{W}$ of $(x', \xi')$ such that the forward broken bicharacteristics from $\tilde{W} \cap \mathcal{H}$ passes through $V$. We will shrink $\tilde{W}$ when necessary, and determine the lens relation on the hyperbolic part of a subset $W' \subset \tilde{W}$, $(x', \xi') \in W'$.

    Smoothly pick $\partial_t \in T\partial M$ around $x'$ such that it is timelike, this also guarantees that $\partial_t$ has consistent time orientation, which is possible since the light cone has been determined. Consider a smooth hypersurface $N \subset \partial M$ containing $x'$ such that $TN$ is spacelike (again possible because light cone has been determined). Locally around $x'$ one can solve an ODE and obtain a smooth temporal function $t$ around $x'$ such that $t|_N = 0$, and we can shrink $\tilde{W}$ such that $t$ is defined in a neighborhood of the projection of $\tilde{W}$ onto $\partial M$. We emphasize that $\partial_t$ is timelike, but we do not know if it is future-pointing or past-pointing. By Lemma \ref{lemma: uniform convergence of broken bichar}, by choosing $W'$ conic neighborhood of $(x', \xi')$ sufficiently small, for any $(y', \eta') \in W' \cap \mathcal{H}$, $L(y', \eta')$ and $L^{-1}(y', \eta')$ must be in $\tilde{W} \subset T^*\tilde{U}$ (here we also use the fact that for directions sufficiently close to a glancing vector, time spent in the interior is close to 0, see the proof of Lemma \ref{lemma: broken bichar converge to gliding on boundary} (1)). Moreover, its forward broken bicharacteristic passes through $V$, so there exists a unique subset $B \subset T^*U$ such that $(y', \eta') \in B \in \text{Domain}(\tilde{L})$. Since $L(y', \eta') \in T^*\tilde{U}$, $L(y', \eta') \in B$ by definition of weak lens relation, same for $L^{-1}(y', \eta')$. Note that any two points in $B$ must satisfy that one of them is in the causal future of the other, which means points in $B$ can be ordered by this. This order can be recovered up to reversing by the value of $t$ from smallest to largest, and $(y', \eta')$ must be in between $L(y', \eta')$ and $L^{-1}(y', \eta')$. If we consistently order them from smallest to largest for any such $(y', \eta')$, we would be able to recover either $L$ or $L^{-1}$.
\end{proof}

We then use the following theorem to determine the jet bundle.
\begin{thm}[\cite{stefanov2024boundary} Theorem 2.1]
    Let $g$ and $\hat{g}$ be two Lorentzian metrics defined near some $x_0 \in \partial M$ so that $\hat{g} = \mu_0 g$ on $T\partial M \times T\partial M$ locally with some $0 < \mu_0 \in C^\infty(\partial M)$. Let $(x_0, v_0) \in T\partial M$ be lightlike for $g$. Assume that $\partial M$ is strictly convex with respect to $g$ in the direction of $(x_0,v_0)$. Assume that either
    \begin{itemize}
        \item[(i)] $\hat{\mathcal{S}}^\# = \mathcal{S}^\#$ in a neighborhood of $(x_0, v_0^b)$, or
        \item[(ii)] $\hat{\mathcal{S}} = \mathcal{S}$ in a neighborhood of $(x_0, v_0)$, and $\mu_0 = \text{const}$.
    \end{itemize}
    Then there exists $\mu(x) > 0$ with $\mu = \mu_0$ on $\partial M$, and a local diffeomorphism $\psi$ near $\partial M$ preserving it pointwise, so that the jets of $g$ and $\mu \psi^*\hat{g}$ coincide on $\partial M$ near $x_0$.
\end{thm}
The $\mathcal{S}^{\sharp}$ in the theorem is the lens relation $L$. Moreover, this deterministic result is independent of whether the lens relation is future-pointing or past-pointing.

\section{Determine Conformal Factor}\label{sec: determine conformal factor}

Suppose both $(g_1, A_1, q_1)$ and $(g_2, A_2, q_2)$ have the same DN map $\Lambda^{U,V}$, by Proposition \ref{prop: same conformal class on boundary} $\bar{g}_1$ and $\bar{g}_2$ share the same $\tilde{U}, \tilde{V}$ and conformal class on $T(\tilde{U} \cup \tilde{V})$. Then for some fixed $h$ in the same conformal class as $\bar{g}_1$ and $\bar{g}_2$, we have $\bar{g}_1 = e^{\varphi_1}h$, $\bar{g}_2 = e^{\varphi_2}h$. From equation \eqref{eq: conformal factor relation}, if a forward broken null-geodesic goes from $z' \in \tilde{U}$ to $x' \in \tilde{V}$, then
\[
(n-2)\varphi_1(z') - n\varphi_1(x') = (n-2)\varphi_2(z') - n\varphi_2(x').
\]
In other words, $\bar{g}_1 =  e^{\varphi}\bar{g}_2$ with $\varphi = \varphi_1 - \varphi_2$, for any forward broken null-geodesic connecting $z' \in U$ and $x' \in V$,
\[
(n-2)\varphi(z') = n\varphi(x').
\]
We will show that this condition actually implies $\varphi$ is locally constant on $\tilde{U} \cup \tilde{V}$. When dimension of $M$ is 3, the main idea is to use light cone in tangent space to approximate the light cone formed by boundary null-geodesics.

\begin{prop}\label{prop: dimension 3 conformal factor locally constant}
    If $M$ has dimension 3, and for any forward null-geodesic from some $z' \in U$ to $x' \in V$, $\varphi$ satisfies $(n-2)\varphi(z') = n\varphi(x')$, then $\varphi$ is locally constant on $\tilde{U} \cup \tilde{V}$.
\end{prop}
\begin{proof}
    The assumption gives the following result: if $\tilde{L}$ is the weak lens relation and $\tilde{L}(B_U) = B_V$, then $\varphi|_{\pi(B_U)} = \frac{c}{n-2}$ and $\varphi|_{\pi(B_V)} = \frac{c}{n}$ for some constant $c$, where $\pi: T^*\partial M \to \partial M$. For $x' \in \tilde{U} \cup \tilde{V}$, let $\gamma^{gl}$ be the projection of a gliding ray that passes through $U, V$ and $x'$. By Lemma \ref{lemma: broken bichar converge to gliding on boundary} (1), for any $z' \in \gamma^{gl} \cap U$, there exists a sequence of broken null-geodesic with initial point $x'$ such that a sequence of boundary points converge to $z'$, so $\varphi(x') = \varphi(z')$. Similarly, for any $y' \in \gamma^{gl} \cap V$, $(n-2)\varphi(x') = n\varphi(y')$. Thus, $\varphi|_{\gamma^{gl} \cap U} = \frac{c}{n-2}$ and $\varphi|_{\gamma^{gl} \cap V} = \frac{c}{n}$ for some constant $c$.

    Locally focus on a small connected neighborhood of some point $x' \in W \subset \tilde{U}$. In dimension 3, by the definition of recoverable point, for any $z' \in W$, both glancing rays of $z'$ passes through $V$. Note that the glancing rays are the same for two metrics in $W$ since they are conformally equivalent. Use the fact that the projection of glancing ray is the null-geodesic of boundary metric $\bar{g}_1$ and $\bar{g}_2$, we thus convert the assumption to $\varphi$ being constant along any boundary null-geodesic. To show $\varphi$ is locally constant at $x'$, suffice to show for any $y'$ sufficiently close to $x'$, a boundary null-geodesic of $y'$ intersects with a boundary null-geodesic of $x'$ in $W$.

    We shall use the light cone in the tangent space to approximate the light cone generated by geodesics. Consider some local coordinate $\phi: W \to N \subset \mathbb{R}^2$, $\phi(x') = 0$, and $h = (\phi^{-1})^*\bar{g}_1$ on $N$. Let $N' \subset N$ be a small neighborhood of 0 such that the null-geodesic of $h$ for any point $p \in N'$ is uniformly close to the straight line passing through $p$ in the same direction in $N$. Then the two null-geodesics of $p$ are uniformly approximated by a cross formed by two straight lines centered at $p$. It is clear that when $N'$ is sufficiently small, the cross centered at $p$ intersects with the cross centered at 0; since the approximation is uniform and the geodesics are connected, we have the union of two null-geodesics from $p$ would intersect the union of two null-geodesics from 0. This proves that for any $y'$ sufficiently close to $x'$, the null-geodesics intersect so $\varphi(y') = \varphi(x')$, hence $\varphi$ is locally constant on $\tilde{U}$. The proof for points in $\tilde{V}$ is the same. Clearly, if $x' \in U$ is connected to $y' \in V$ by a forward null-geodesic, then the constant $c_1$ around $x'$ and the constant $c_2$ around $y'$ need to satisfy $(n-2)c_1 = nc_2$.
\end{proof}
When the dimension is higher, for every point in the recoverable set, we have an open set of recoverable directions in $\mathcal{G}$. We can not simply use the same argument as in dimension 3, because the set of null-geodesics do not naturally form a cross anymore.

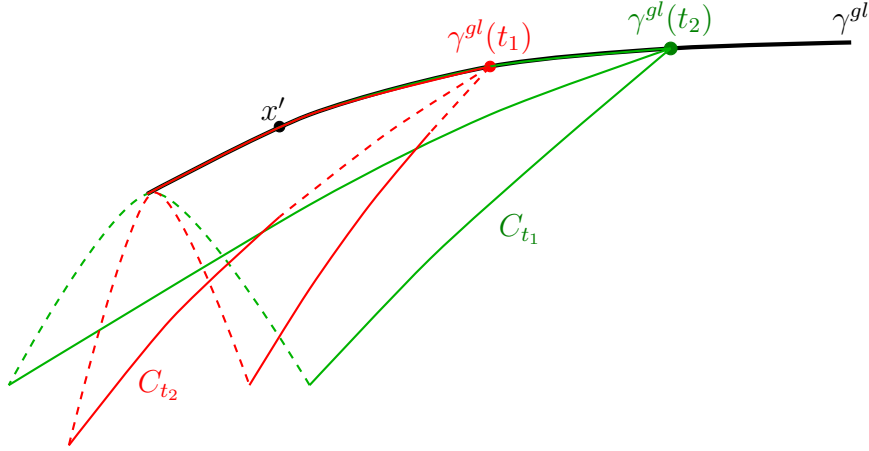
\begin{figure}[h]

    \begin{tikzpicture}[scale=8]

        \draw[ultra thick, black] plot[smooth] coordinates
        {(0.13,0.118) (0.4,0.25) (0.7,0.33) (1.0,0.36) (1.3,0.37)};
        \node[black] at (1.3,0.41) {$\gamma^{gl}$};

        \filldraw[red] (0.7,0.33) circle (0.25pt);
        \node[red] at (0.7,0.38) {$\gamma^{gl}(t_1)$};

        \filldraw[green!50!black] (1.0,0.36) circle (0.3pt);
        \node[green!50!black] at (1.0,0.41) {$\gamma^{gl}(t_2)$};

        \filldraw[black] (0.35,0.23) circle (0.25pt);
        \node at (0.34,0.26) {$x'$};

        \draw[thick, green!70!black] plot[smooth] coordinates
        {(0.13,0.118) (0.4,0.25) (0.7,0.33) (1.0,0.36)};

        \draw[thick, green!70!black] plot[smooth] coordinates
        {(1.0,0.36) (0.7,0.25) (0.4,0.1) (-0.1, -0.2)};

        \draw[thick, green!70!black] plot[smooth] coordinates
        {(1.0,0.36) (0.8,0.19) (0.6,0.01) (0.4, -0.2)};

        \draw[thick, green!70!black, dashed] plot[smooth] coordinates
        {(0.4, -0.2) (0.14,0.12) (-0.1, -0.2)};
        \node[green!50!black] at (0.75,0.06) {$C_{t_1}$};

        \draw[thick, red] plot[smooth] coordinates
        {(0.13,0.118) (0.4,0.25) (0.7,0.33)};
        \draw[thick, red, dashed] plot[smooth] coordinates
        {(0.7,0.33) (0.53,0.22) (0.35,0.08)};
        \draw[thick, red] plot[smooth] coordinates
        {(0.35,0.08) (0.17,-0.09) (0,-0.3)};

        \draw[thick, red, dashed] plot[smooth] coordinates
        {(0.7,0.33) (0.6, 0.22)  };

        \draw[thick, red] plot[smooth] coordinates
        {(0.6, 0.22)  (0.5, 0.1)  (0.4, -0.04) (0.3, -0.2)};
        \node[red] at (0.15,-0.2) {$C_{t_2}$};

        \draw[thick, red, dashed] plot[smooth] coordinates
        {(0,-0.3) (0.14,0.12) (0.3, -0.2)};

    \end{tikzpicture}
    \caption{Dimension at least 4: in $\tilde{U}$, we use that fact that $\varphi$ would stay constant along any piecewise boundary null-geodesic curve where each piece of boundary null-geodesic when fully extended passes through $V$. We show that the union of such paths from $x'$ has non-empty interior. Same argument works for $\tilde{V}$.}
    \centering
\end{figure}

\begin{prop}\label{prop: dimension 4 and higher conformal factor locally constant}
    If $M$ has dimension at least 4, and for any forward null-geodesic from some $z' \in U$ to $x' \in V$ satisfies $(n-2)\varphi(z') = n\varphi(x')$, then $\varphi$ is locally constant on $\tilde{U} \cup \tilde{V}$.
\end{prop}
\begin{proof}
Again the first part is the same as in dimension 3, we have for any boundary null-geodesic $\gamma^{gl}$ passing through both $U$ and $V$, $\varphi|_{\gamma^{gl} \cap U} = \frac{c}{n-2}$ and $\varphi|_{\gamma^{gl} \cap V} = \frac{c}{n}$ for some constant $c$.

From now on, we look at the boundary manifold $(\partial M, \bar{g})$, for example when we say null-geodesic we refer to the geodesic in the boundary with respect to $\bar{g}$. Fix some $x' \in \tilde{U}$, there exists some $(x', \xi')$ recoverable direction. We will be working in a sufficiently small geodesically convex neighborhood $W$ of $x'$, and use the fact that any two points in $W$ that are connected by a null-geodesic can not be connected by a piecewise lightlike path with fixed time orientation that is not a null-geodesic (because otherwise they would be connected by a timelike geodesic).


Let $\gamma^{gl}$ be a null-geodesic passing through $x'$ corresponding to a recoverable direction. Denote $C_t$ the union of inextendible null-geodesics in $W$ that passes through $\gamma^{gl}(t)$, and will pass through $V$ when fully extended, denote $\mathcal{C}(x') := \cup C_t$. For $t_1 \neq t_2$, suppose there exists some $y' \in C_{t_1} \cap C_{t_2}$ such that $y' \notin \gamma^{gl}$. Then any two points of $y'$, $\gamma^{gl}(t_1)$ and $\gamma^{gl}(t_2)$ are connected by some null-geodesic. The three are obviously not on the same null-geodesic since $y' \notin \gamma^{gl}$, so it contradicts geodesic convexity of $W$. As a result, we have for any $t_1 \neq t_2$, $C_{t_1} \cap C_{t_2}  = \gamma^{gl}$ is of dimension 1. On the other hand, we know $C_t$ is an $n-2$ dimensional submanifold for all $t \neq 0$, and it changes continuous as $t$ varies. Thus as $t$ goes from some $t_1$ to $t_2$, $C_t$ changes continuously from an $n-2$ dimensional submanifold to another $n-2$ dimensional submanifold such that at any two distinct time, the two submanifolds only intersect at $\gamma^{gl}$. This implies $\mathcal{C}(x')$ has non-empty interior (since $\partial M$ has dimension $n-1$). Moreover, $\varphi$ is constant along $\gamma^{gl}$ and on each $C_t$, so $\varphi|_{\mathcal{C}(x')}$ is constant.

We can similarly construct $\mathcal{C}(z')$ for any $z'$ in a neighborhood of $x'$. Note that $\mathcal{C}(z')$ changes continuously and has non-empty interior. Hence for any $z'$ sufficiently close to $x'$, $\mathcal{C}(z') \cap \mathcal{C}(x') \neq \emptyset$, implying $\varphi$ is constant around $x'$. Same proof works for points in $\tilde{V}$.

\end{proof}

\section{Determine 1-Form}\label{sec: determine 1-form}

Now we use \eqref{eq: 1-form relation} to recover the 1-form on $\tilde{U} \cup \tilde{V}$, assuming the conformal class of $\bar{g}$ has been recovered on $\tilde{U} \cup \tilde{V}$. As in the proof of Proposition \ref{prop: same conformal class on boundary}, for every recoverable direction $(x', \xi') \in T^*\tilde{U} \cap \mathcal{G}$, one can identify its corresponding forward gliding ray $\Gamma^{gl}$ on both $T^*U$ and $T^*V$ using weak lens relation (for details, see proof of the proposition). Denote $\gamma^{gl}$ the projection of $\Gamma^{gl}$, let $\phi: [-1, 1] \to \gamma^{gl}$ be a parametrization of $\gamma^{gl}$ around $x'$ such that $\phi(0) = x'$ and $\phi([-1, 1]) \subset \tilde{U}$. Since the conformal class of $\bar{g}$ is already recovered, one can recover $(\xi')^{\sharp}$ up to a positive scaling. This means we can determine the direction of $\dot{\gamma}^{gl}(0)$, since up to positive rescaling, $\dot{\gamma}^{gl}(0) = -2 (\xi')^{\sharp}$, see \eqref{eq: symbol of P} and Section \ref{subsec: geometry}. We may thus assume $\phi'(0)$ is in the same direction as $\dot{\gamma}^{gl}(0)$, otherwise just reverse the parametrization $\phi$. We show that for any $t \in [-1, 1]$, we can explicitly construct
\[
\exp \left( i\int_0^t A(\phi'(s))ds \right).
\]

Fix some $t$, let $z' = \phi(t)$ and $(z', \zeta')$ the corresponding point on $\Gamma^{gl}$. By Proposition \ref{prop: upgrade weak lens relation}, the lens relation $L$ can be recovered up to time orientation on a conic neighborhood of glancing set. Then for each point of $\Gamma^{gl}((x', \xi') \to (z', \zeta'))$, we can find neighborhood of uniform size on which $L$ is known. Let $(x'_j, \xi'_j) \in T^*U \cap \mathcal{H}$ such that $(x'_j, \xi'_j) \to (x', \xi')$, by Lemma \ref{lemma: uniform convergence of broken bichar}, the forward broken bicharacteristic $\Gamma^b_j$ from $(x'_j, \xi'_j)$ converges uniformly to $\Gamma^{gl}$ on bounded time interval. By Lemma \ref{lemma: broken bichar converge to gliding on boundary} (1), we can find a sequence of $(z'_j, \zeta'_j)$ on $\Gamma^b_j$ such that $(z'_j, \zeta'_j) \to (z', \zeta')$. As a result, for $j$ sufficiently large, all the boundary points of $\Gamma^b_j$ between $(x'_j, \xi'_j)$ and $(z'_j, \zeta'_j)$ are in the neighborhood where $L$ is known up to time orientation. That is, let $k_j$ be such that $(z'_j, \zeta'_j) = L^{k_j}(x'_j, \xi'_j)$, then one can recover $k_j$ up to a sign determined by time orientation.

Denote $(y'_j, \eta'_j)$ some arbitrary point on $\Gamma^b_j \cap T^*V$, by equation \eqref{eq: 1-form relation} we know
\[
\exp \left( i\int_{\gamma^b_j(x'_j \to y'_j)} A(\dot{\gamma}^b_j) \right) = \pm i (-1)^{l_j}e^{i\theta_j}, \quad \exp \left( i\int_{\gamma^b_j(z'_j \to y'_j)} A(\dot{\gamma}^b_j) \right) = \pm i (-1)^{m_j}e^{i\omega_j}
\]
where $|l_j - m_j| = |k_j|$ is known, $e^{i\theta_j}$ and $e^{i\omega_j}$ are the argument terms that can be explicitly computed. In particular, this means we can explicitly compute the quotient
\[
\exp \left( i\int_{\gamma^b_j(x'_j \to z'_j)} A(\dot{\gamma}^b_j) \right) = (-1)^{|k_j|}e^{i(\theta_j - \omega_j)}.
\]
Note that the uncertainty caused by time orientation is eliminated after taking the quotient. By Lemma \ref{lemma: uniform convergence of broken bichar}, the following limit converges and can be computed:
\[
\lim_{j \to \infty}(-1)^{|k_j|}e^{i(\theta_j - \omega_j)} = \exp \left( i\int_{\gamma^{gl}(x' \to z')} A(\dot{\gamma}^{gl}) \right) = \exp \left( i\int_0^t A(\phi'(s))ds \right).
\]
Since $t \in [-1, 1]$ is arbitrary, we can then take derivative and obtain $A(\phi'(0))$. Moreover, we can compute this for any gliding ray that goes from $U$ to $V$. By assumption of recoverable set, this means 2 gliding rays for dimension 3, and an open set of glancing directions for higher dimensions. In both cases, one can construct $A(v)$ for a set of $v$ that spans the entire $T_{x'}\tilde{U}$. The construction for points in $\tilde{V}$ is the same, in particular we proved the following result.

\begin{prop}\label{prop: determine 1-form}
    Given the DN map $\Lambda^{U,V}_{g,A,q}$, let $T\tilde{U} \cup T\tilde{V} \subset T\partial M$ be the set of tangential vectors on $\tilde{U} \cup \tilde{V}$, then $A|_{T\tilde{U} \cup T\tilde{V}}$ can be computed.
\end{prop}

\section{Proof of Main Theorems}\label{sec: proof of main theorems}

\begin{proof}[Proof of Theorem \ref{theorem: reconstruction}:]
    By Proposition \ref{prop: same weak lens relation}, one can construct the weak lens relation $\tilde{L}$. Then by Proposition \ref{prop: same conformal class on boundary}, the recoverable sets $\tilde{U} \cup \tilde{V}$ can be determined, and the conformal class of the boundary metric $\bar{g} = g|_{T\partial M \times T\partial M}$ can be constructed on $\tilde{U} \cup \tilde{V}$. Moreover, by Proposition \ref{prop: upgrade weak lens relation}, we can recover the lens relation $L$ up to time orientation in a neighborhood of the glancing set. Finally, by Proposition \ref{prop: determine 1-form}, $A|_{T(\tilde{U}\cup \tilde{V})}$ can be constructed from the principal symbol of $\Lambda^{U,V}_{g,A,q}$.
\end{proof}

\begin{proof}[Proof of Theorem \ref{theorem: determination}:]
    By Theorem \ref{theorem: reconstruction}, the two metrics share the same recoverable sets $\tilde{U}$ and $\tilde{V}$, the same conformal class on recoverable sets, and the same lens relation around glancing sets up to time orientation. By \cite[Theorem 2.1]{stefanov2024boundary}, this is enough to determine the jet bundle of the metric on the recoverable sets, because choice of the time orientation does not affect the metric. Finally, by Proposition \ref{prop: dimension 3 conformal factor locally constant} and Proposition \ref{prop: dimension 4 and higher conformal factor locally constant}, the conformal factor satisfies the relation stated in the theorem.
\end{proof}

We now consider the case where the metric on $U$ or $V$ is a priori given. In this case, we show that the boundary metric can be explicitly constructed on a larger set than the recoverable set.
\begin{proof}[Proof of Theorem \ref{theorem: reconstruction with a priori info}:]
    First assume the boundary metric $\bar{g}$ is given on $\tilde{U}_0$. Let $y' \in \tilde{V}_0$, recall that by definition this means there exists some $x' \in \tilde{U}_0$ and $\gamma^b$ broken null-geodesic such that $x'$ and $y'$ are connected by $\gamma^b$. Then there exists some $\Gamma^b$ broken null-bicharacteristic such that the projection of $\Gamma^b$ is $\gamma^b$, and $\Gamma^b$ connects $(x', \xi')$ and $(y', \eta')$ for some $\xi' \in T^*_{x'}\partial M$ and $\eta' \in T^*_{y'}\partial M$. By Proposition \ref{prop: principal symbol formula}, choose some $f$ whose wavefront set contains $(x', \xi')$, the absolute value of the principal symbol gives
    \[
    |\bar{g}(y')|_{y'}^{-1/4}|\eta_n|^{1/2} = |\bar{g}(x')|_{x'}^{-1/4}|\xi_n|^{-1/2}\frac{\sigma_{y'}(\Lambda^{U,V}_{g,A,q}f)|_{(y', \eta')}}{(L^{-k})^*\sigma_{x'}(f)|_{(x', \xi')}},
    \]
    where
    \[
    |\xi_n| = \sqrt{-\bar{g}^{\alpha\beta}\xi_{\alpha}\xi_{\beta}} = (-\bar{g}(\xi', \xi'))^{1/2}, \quad |\eta_n| = \sqrt{-\bar{g}^{\alpha\beta}\eta_{\alpha}\eta_{\beta}} = (-\bar{g}(\eta', \eta'))^{1/2}.
    \]
    Note that the right hand side is fully computable since we are given $\bar{g}$ on $\tilde{U}_0$. Thus we can compute
    \[
    C(y', \eta') := |\bar{g}(y')|_{y'}^{-1}\bar{g}(\eta', \eta')
    \]
    for any $(y', \eta')$ such that the corresponding backward broken null-bicharacteristic passes through $U$. In particular, if this holds for $(y', \eta')$, then it holds for a conic neighborhood of $(y', \eta')$. Note that for any $(y', \zeta')$ in that open set, they share the same $|\bar{g}(y')|_{y'}$. Choose a basis $\zeta^1, \cdots, \zeta^{n-1}$ in the open neighborhood, we may assume the determinant $|\bar{g}(y')|_{y'}$ is with respect to this basis. Since $\zeta^{\alpha}+\zeta^{\beta}$ is also in the neighborhood for any $\alpha, \beta = 1 , \dots, n-1$. We thus have
    \[
    C(y', \zeta^{\alpha}) = |\bar{g}(y')|_{y'}^{-1}\bar{g}(\zeta^\alpha, \zeta^\alpha), \quad C(y', \zeta^\alpha+\zeta^\beta) = |\bar{g}(y')|_{y'}^{-1}\bar{g}(\zeta^\alpha+\zeta^\beta, \zeta^\alpha+\zeta^\beta).
    \]
    The above information is enough for us to construct the matrix
    \[
    C(y') := |\bar{g}(y')|^{-1}_{y'}\bar{g}^{-1}(y').
    \]
    As a result, for all $n \geq 3$, the metric at $y'$ can be explicitly computed as
    \[
    \bar{g}(y') = \left(\det(C(y'))^{-\frac{1}{n}}C(y')\right)^{-1}.
    \]

    The same method works when the metric is given on $\tilde{V}_0$. Note that in this case we can similarly construct
    \[
    C(x') := |\bar{g}(x')|_{x'}\bar{g}^{-1}(x'),
    \]
    and so the metric is given by
    \[
    \bar{g}(x') = \left( \det(C(x'))^{\frac{1}{n-2}} C(x') \right)^{-1}.
    \]
    Finally, the fact that time orientation can be determined is because $(y', \eta')$ and $(x', \xi')$ have the same time orientation if they are connected by a broken bicharacteristic.
\end{proof}

\bibliographystyle{abbrv}
\bibliography{main.bib}

\end{document}